\documentclass{amsart}
\usepackage[justification=centering]{caption}
\usepackage{amsmath}
\usepackage{mathrsfs}
\usepackage{amsfonts}
\usepackage{mathtools}
\usepackage{xcolor}
\usepackage{algorithm}
\usepackage{algorithmic}
\usepackage{latexsym}
\usepackage{booktabs}
\usepackage[toc,page]{appendix}
\usepackage{textcomp}
\usepackage{subfigure}

\newcommand{\Rni}{\mathbb{R}^{n_{i}}}

\newcommand{\nfR}{\mathbb{R}}
\newcommand{\xpi}{x_i}

\newcommand{\xmi}{x_{-i}}
\newcommand{\allx}{x}
\newcommand{\fpi}{f_{i}}

\newcommand{\lip}{\left<}
\newcommand{\rip}{\right>}

\newcommand{\idl}{\mbox{Ideal}}
\newcommand{\qmod}{\mbox{Qmod}}
\newcommand{\cvx}{\mbox{conv}}
\newcommand{\ddd}{,\ldots,}

\DeclareMathOperator{\Rank}{rank}

\newcommand{\re}{\mathbb{R}}
\newcommand{\cpx}{\mathbb{C}}

\newcommand{\N}{\mathbb{N}}

\def\af{\alpha}

\newcommand{\st}{\mathit{s.t.}}
\newcommand{\reff}[1]{(\ref{#1})}

\newcommand{\lmd}{\lambda}
\newcommand{\pt}{\partial}
\newcommand{\dt}{\delta}
\newcommand{\nn}{\nonumber}
\newcommand{\eps}{\epsilon}

\newcommand{\mbf}[1]{\mathbf{#1}}
\newcommand{\mcal}[1]{\mathcal{#1}}

\def\rank{\mbox{rank}}

\newcommand{\bdes}{\begin{description}}
\newcommand{\edes}{\end{description}}

\newcommand{\bal}{\begin{align}}
\newcommand{\eal}{\end{align}}

\newcommand{\bnum}{\begin{enumerate}}
\newcommand{\enum}{\end{enumerate}}

\newcommand{\bit}{\begin{itemize}}
\newcommand{\eit}{\end{itemize}}

\newcommand{\bea}{\begin{eqnarray}}
\newcommand{\eea}{\end{eqnarray}}
\newcommand{\be}{\begin{equation}}
\newcommand{\ee}{\end{equation}}

\newcommand{\baray}{\begin{array}}
\newcommand{\earay}{\end{array}}

\newcommand{\bsry}{\begin{subarray}}
\newcommand{\esry}{\end{subarray}}

\newcommand{\bca}{\begin{cases}}
\newcommand{\eca}{\end{cases}}

\newcommand{\bcen}{\begin{center}}
\newcommand{\ecen}{\end{center}}

\newcommand{\bbm}{\begin{bmatrix}}
\newcommand{\ebm}{\end{bmatrix}}

\newcommand{\btab}{\begin{tabular}}
\newcommand{\etab}{\end{tabular}}

\numberwithin{equation}{section}

\theoremstyle{definition}
\newtheorem{definition}{Definition}[section]

\theoremstyle{plain}

\newtheorem{prop}[definition]{Proposition}
\newtheorem{corollary}[definition]{Corollary}

\newtheorem{theorem}[definition]{Theorem}

\newtheorem{example}[definition]{Example}
\newtheorem{alg}[definition]{Algorithm}
\theoremstyle{remark}
\newtheorem{remark}[definition]{Remark}

\numberwithin{equation}{section}

\begin{document}

\title{Nash Equilibrium Problems of Polynomials}
\author[Jiawang Nie]{Jiawang~Nie}
\address{Jiawang Nie, Department of Mathematics,
University of California San Diego,
9500 Gilman Drive, La Jolla, CA, USA, 92093.}
\email{njw@math.ucsd.edu}

\author[Xindong Tang]{Xindong~Tang}
\address{Xindong Tang, Department of Applied Mathematics, The Hong Kong Polytechnic University,
Hung Hom, Kowloon, Hong Kong.}
\email{xindong.tang@polyu.edu.hk}


\begin{abstract}
This paper studies Nash equilibrium problems that
are given by polynomial functions.
We formulate efficient polynomial optimization problems
for computing Nash equilibria. The Moment-SOS relaxations
are used to solve them. Under generic assumptions,
the method can find a Nash equilibrium if there is one.
Moreover, it can find all Nash equilibria
if there are finitely many ones of them.
The method can also detect nonexistence if there is no Nash equilibrium.
\end{abstract}

\maketitle

\section{Introduction}

The Nash equilibrium problem (NEP) is a kind of games for finding strategies
for a group of players such that each player's objective is optimized,
for given other players' strategies.
Suppose there are $N$ players, and the $i$th player's strategy
is the variable $x_{i} \in \re^{n_{i}}$ (the $n_i$-dimensional real Euclidean space).
We denote that
\[
x_i\,  \coloneqq  \, (x_{i,1}, \ldots, x_{i,n_i}), \quad
\allx \,  \coloneqq  \, (x_1, \ldots, x_N).
\]
The total dimension of all players' strategies is
\[
n : \, = \, n_1 + \cdots + n_N.
\]
When the $i$th player's strategy $x_i$ is being optimized,
we use $x_{-i}$ to denote the subvector
of all players' strategies except $x_i$, i.e.,
\[
x_{-i} \,  \coloneqq  \, (x_1, \ldots, x_{i-1}, x_{i+1}, \ldots, x_N),
\]
and write $\allx = (x_i, x_{-i})$ accordingly.
When the writing $x_{-i}$ appears,
 the $i$th player's strategy is being considered for optimization,
while the vector of all other players' strategies is fixed to be $\xmi$.
In an NEP, the $i$th player's best strategy $x_i$
is the minimizer for the optimization problem
\be
\label{eq:subprob}
\mbf{F}_i(x_{-i}):
\left\{ \begin{array}{cl}
\min\limits_{\xpi\in \Rni}  &  \fpi(\xpi,\xmi) \\
\st & g_{i,j}(\xpi) = 0 \ (j \in \mcal{E}_i), \\
    & g_{i,j}(\xpi) \geq 0 \ (j \in \mcal{I}_i),
\end{array} \right.
\ee
for the given other players' strategies $\xmi$.
In the above, $f_i$ is the $i$th player's objective function,
and $g_{i,j}$ are constraining functions in $\xpi$.
The $\mcal{E}_i$ and $\mcal{I}_i$ are disjoint labeling sets
of finite cardinalities (possibly empty).
The feasible set of the optimization $\mbf{F}_i(x_{-i})$ in \reff{eq:subprob} is
\be \label{eq:feaset}
X_i \,  \coloneqq  \,
\{x_i \in \Rni: \, g_{i,j}(\xpi) = 0\, (j \in \mcal{E}_i), \,
g_{i,j}(\xpi) \geq 0 \, (j \in \mcal{I}_i)  \} .
\ee
For NEPs, each set $X_i$ does not depend on $x_{-i}$.
This is different from generalized Nash equilibrium problems (GNEPs),
where each player's feasible set depends on other players' strategies.
We say the strategy vector $x$ is feasible if
\[
 \allx   = (x_1, \ldots, x_N) \, \in \, X \,  \coloneqq  \,
X_1 \times \cdots \times X_n.
\]
That is, each $x_i \in X_i$.
The NEP can be formulated as
\be
\label{eq:NEP}
\mbox{find} \quad \allx^*  \in\nfR^{n}
\quad \mbox{such that each}
\,\,\xpi^* \,\, \mbox{is a minimizer of} \, \mbf{F}_i(x_{-i}^*),
\ee
where $\allx^* = (x_1^*, \ldots, x_N^*)$.
A solution of \reff{eq:NEP} is called a {\it Nash equilibrium} (NE)\footnote{
In some literature, this is also called a {\it pure strategy Nash equilibrium},
in contrast to {\it mixed strategy Nash equilibria},
which are probability measures supported on the feasible strategy sets.
We refer to Section~\ref{sc:conc} for more details on mixed strategy NEs;
also see \cite{dresher2016polynomial,kroupa2023multiple,Nash1951,
parrilo2006polynomial,stein2008separable,Young2014}.
}.
When the defining functions $f_i$ and $g_{i,j}$ are continuous,
then the NEP is called a {\it continuous Nash equilibrium problem}.
In this paper, we consider cases that each $f_i$ is a polynomial in
$\allx$ and $g_{i,j}$'s are polynomials in $x_i$.
Such an NEP is called a {\it Nash equilibrium problem of polynomials} (NEPP).
The following is an example.
\begin{example}
\label{ep:isolatedep}  \rm
Consider the $2$-player NEP with the individual optimization
\[
\baray{l}
\mbox{1st player:}  \left\{
  \begin{array}{rl}
    \min\limits_{x_1 \in \re^2} &
     x_{1,1}(x_{1,1}+x_{2,1}+4x_{2,2})+2x_{1,2}^2 \\
 \st & 1-(x_{1,1})^2-(x_{1,2})^2 \ge 0,
  \end{array}\right.
\\
\mbox{2nd player:}   \left\{
  \begin{array}{rl}
    \min\limits_{x_2 \in \re^2} & x_{2,1}(x_{1,1}+2x_{1,2}+x_{2,1})+x_{2,2}(2x_{1,1}+x_{1,2}+x_{2,2}) \\
    \st & 1 -(x_{2,1})^2 -(x_{2,2})^2 \ge 0.
  \end{array}\right.
\earay
\]
In this NEP, each player's objective is strictly convex with respect to its strategy,
because their Hessian matrices with respect to their own strategies are positive definite.
This NEP has only $3$ NEs (see Section~\ref{ssc:convex}), which are
\[
\boxed{
\begin{array}{rll}
\mbox{1st NE:} & x_1^* = (0, 0), & x_2^* = (0, 0); \\ \hline
\mbox{2nd NE:} & x_1^* = (1, 0), & x_2^* = \frac{1}{\sqrt{5}} (-1, -2 ) ;\\ \hline
\mbox{3rd NE:} & x_1^* = (-1, 0), & x_2^* = \frac{1}{\sqrt{5}}(1, 2 ) .
\end{array} }
\]
\end{example}

NEPs are challenging problems to solve.
Even for the special cases where each player's objective function
is multilinear in $(x_1\ddd x_N)$, and each feasible set is a simplex,
finding an NE is PPAD-complete \cite{DasGoldPap09}.
The problem becomes more difficult when players' optimization problems are nonconvex.
This is because an NE $\allx^* = (x_1^*, \ldots, x_N^*)$ requires that each $x_i^*$
is a global minimizer of $\mbf{F}_i(x_{-i}^*)$.
Indeed, finding a global minimizer of a single polynomial optimization problem
is already NP-hard \cite{LasBk15}.
For polynomial optimization problems, global optimizers can be computed efficiently
by the Moment-SOS hierarchy of semidefinite relaxations
(see \cite{Las01,LasBk15,Lau09} for related work).
Moreover, for some NEPs, there may not exist any NE.
Such NEPs are also interesting and have important applications
(e.g., NEPs in generative adversarial networks \cite{Farnia2020}).
If an NE does not exist, how can we detect its nonexistence?
This question is mostly open for general NEPs,
to the best of the author's knowledge.
However, under certain nonsingularity conditions,
nonexistence of NEs for NEPPs can be certified by the infeasibility
of some semidefinite programs.
For the above reasons, this paper focuses on NEPPs.

NEPs are generalizations of {\it finite games} \cite{Nash1951},
where each $X_i$ is a finite set, i.e., $|X_i|<\infty$.
In recent years, there has been an increasing
number of applications of NEPs in various fields,
such as economics,
environmental protection,
politics,
supply chain management,
machine learning, etc.
We refer to \cite{breton2006game,Contreras2004,Farnia2020,Goodfellow2014,maskin1999,Schofield2002}
for some recent applications of NEPs.
In Section~\ref{sc:num}, we present some concrete applications of NEPs
in environmental pollution control and the electricity market.
Moreover, we refer to surveys \cite{Aubin13,Young2014} for more general work on NEPs.

In this paper, our primary goal is to find NEs for NEPs.
In the following, we review some previous work on solving NEPs.
The NEP is called a {\it zero-sum game} if the sum of objective functions
is identically equal to a constant.
Two-player zero-sum games are equivalent to
{\it saddle point problems}.
We refer to \cite{CLOu14,Nedic2009} for algorithms of
solving saddle point problems under convexity assumptions, and \cite{Nie2018saddle}
for the method of
solving nonconvex polynomial saddle point problems.
For finite games, finding mixed strategy solutions is a special case of NEPs of polynomials;
see \cite{Ahmadi2020,Datta10,Kontogiannis2006,Young2014} for some related approaches.
There exists work on mixed strategy solutions for continuous games,
see \cite{dresher2016polynomial,parrilo2006polynomial,stein2008separable}
for mixed strategy solutions to polynomial games,
and \cite{adam2021double,kroupa2023multiple}
for the recently developed multiple oracle algorithms.
For finding pure strategy solutions for general continuous NEPs,
we refer to techniques such as variational inequalities \cite{GurPang09,Kulkarni2012},
Nikaido-Isoda functions \cite{krawczyk2000,UryRub94},
and manifold optimization tools \cite{Ratliff2013}.
In most earlier work, convexity is often assumed for each player's optimization.
Moreover, NEPs are special cases of GNEPs \cite{FacKan10OR},
where each player's feasible set is dependent on other players' strategies.
For GNEPs given by polynomial functions,
the work \cite{Couzoudis2013} introduces a parametric SOS relaxation approach,
and the Gauss-Seidel method using Moment-SOS relaxations is studied in \cite{Nie2020gs}.
When the GNEPs are further assumed to be convex, the semidefinite relaxation method
is introduced in \cite{Nie2023convex}.
At the moment, it is mostly an open question to solve general NEPs,
especially when the players' optimization problems are nonconvex.

\subsection*{Contributions}

This paper focuses on Nash equilibrium problems that are given by polynomials.
We formulate efficient polynomial optimization for computing
one or more Nash equilibria.
The Moment-SOS hierarchy of semidefinite relaxations is used
to solve the appearing polynomial optimization problems.
Our major results are:

\begin{itemize}
\item Under some genericity assumptions, we prove that our method can compute
a Nash equilibrium if there exists one, or it can detect nonexistence of NEs.
Moreover, if there are only finitely many NEs, we show how to find all of them.
In the prior existing work, there do not exist similar methods that can
achieve such computational goals.

\item When the objective and constraining polynomials are generic
(i.e., they have generic coefficients), we show that the NEPP
has only finitely many KKT points.
For such generic NEPPs, our method can compute all NEs,
if they exist, or can detect their nonexistence.

\item When the objective and constraining polynomials are not generic,
our method can still be applied to compute one or more NEs,
or to detect their nonexistence.
Even if there are infinitely many NEs, our method may still be able to get an NE.
In computational practice, there is no need to check if the NEP is generic
or not to implement our algorithms.
In fact, our method is self-verifying, that in the actual implementation,
the algorithm can check whether the computed point is an NE,
and check if the computed solution set is complete or not.

\end{itemize}

%
%

The paper is organized as follows.
Some preliminaries about polynomial optimization
are given in Section~\ref{sc:pre}.
We give efficient polynomial optimization formulations
in Section~\ref{sc:finding}.
We show how to solve polynomial optimization problems
by the Moment-SOS hierarchy in Section~\ref{sc:solvepop}.
Numerical experiments and applications are given in Section~\ref{sc:num}.
Conclusions and discussions are proposed in Section~\ref{sc:conc}.
The finiteness of the KKT set for generic NEPs is showed in Appendix.

\section{Preliminaries}
\label{sc:pre}

\subsection*{Notation}
The symbol $\mathbb N$ (resp., $\mathbb R$, $\mathbb C$) stands
for the set of nonnegative integers (resp., real numbers, complex numbers).
For a positive integer $k$, denote the set $[k]  \coloneqq  \{1, \ldots, k\}$.
For a real number $t$, $\lceil t \rceil$ (resp., $\lfloor t \rfloor$)
denotes the smallest integer not smaller than $t$
(resp., the biggest integer not bigger than $t$).
For the $i$th player's strategy variable $x_i \in \re^{n_i}$, the $x_{i,j}$
denotes the $j$th entry of $x_i$, $j = 1, \ldots, n_i$.
The $\nfR[\allx]$ (resp., $\cpx[\allx]$)
denotes the ring of polynomials with real (resp., complex) coefficients in $\allx$.
The $\nfR[\allx]_d$ (resp., $\cpx[\allx]_d$)
denotes its subset of polynomials whose degrees are not greater than $d$.
For the $i$th player's strategy vector $x_i$,
the notation $\re[x_i]$, $\cpx[x_i]$, $\re[x_i]_d$, $\cpx[x_i]_d$
are defined in the same way.
For $i$th player's objective $f_i(x_i,x_{-i})$,
the notation $\nabla_{x_i}f_i$, $\nabla^2_{x_i}f_i$
respectively denote its gradient and Hessian with respect to $x_i$.

In the following, we use the letter $z$ to represent either $\allx$ or $x_i$
for the convenience of discussion. Suppose $z  \coloneqq  (z_1,\ldots,z_l)$
and $\af  \coloneqq  (\af_1, \ldots, \af_l) \in \N^{l}$, denote
\[
z^\alpha  \coloneqq  z_1^{\alpha_1} \cdots z_l^{\alpha_l}, \quad
|\alpha| \coloneqq \alpha_1+\ldots+\alpha_l.
\]
For an integer $d >0$, denote the monomial power set
\[
{\mathbb{N}}_d^l \,  \coloneqq  \,
\{\alpha\in {\mathbb{N}}^l: \, \ |\alpha| \le d \}.
\]
We use $[z]_d$ to denote the vector of all monomials in $z$
and whose degree is at most $d$, ordered in the graded alphabetical ordering.
For example, if $z =(z_1, z_2)$, then
\[
[z]_3 = (1,  z_1, z_2, z_1^2, z_1z_2,
z_2^2, z_1^3, z_1^2z_2, z_1z_2^2, z_2^3).
\]

Throughout the paper, the word “generic” is used for a property
if it holds for all points outside a set of Lebesgue measure zero in the space of input data.
For a given multi-degree $(d_1,\ldots,d_N)$ (resp., a degree $d$)
in the variable $\allx=(x_1,\ldots,x_N)$ (resp., in variable $x_i$),
we say a polynomial $p(\allx)$ (resp., $q(x_i)$)
is generic if the coefficient vector of $p$ (resp., $q$)
is generic in the space of coefficients.
For multi-degrees $a_1,\ldots,a_N$ and degrees $b_{1,1},b_{1,2},\ldots,b_{1,m_1},b_{2,1},\ldots,b_{N,m_N}$,
we say the NEPP is generic if for each $i$ and $j$,
the $f_i(x_1\ddd x_N)$ is a generic polynomial with multi-degree $a_i$,
and the $g_{i,j}(x_i)$ is a generic polynomial whose degree is $b_{i,j}$.

\subsection{Ideals and positive polynomials}
\label{sc:idqm}
Let $\mathbb{F}=\re\ \mbox{or}\ \cpx$. For a polynomial $p \in\mathbb{F}[z]$
and subsets $I, J \subseteq \mathbb{F}[z]$, define the product and Minkowski sum
\[
p \cdot I   \coloneqq \{ p  q: \, q \in I \}, \quad
I+J   \coloneqq  \{a+b: \, a \in I, b \in J  \}.
\]
The subset $I$ is an ideal if $p \cdot I \subseteq I$ for all $p\in\mathbb{F}[z]$
and $I+I \subseteq I$.
For a tuple of polynomials $q = (q_1, \ldots, q_m)$, the set
\[
\idl[q] \coloneqq  q_1\cdot\mathbb{F}[z] + \ldots + q_m \cdot \mathbb{F}[z]
\]
is the ideal generated by $q$, which is the smallest ideal containing each $q_i$.

We review basic concepts in polynomial optimization.
A polynomial $\sigma \in \re[z]$ is said to be a sum of squares (SOS) if $\sigma=s_1^2+s_2^2+\ldots+s_t^2$ for some polynomials $s_1,\ldots,s_t\in\nfR[z]$.
The set of all SOS polynomials in $z$ is denoted as $\Sigma[z]$.
For a degree $k$, we denote the truncation
\[
\Sigma[z]_{2k} \, \coloneqq \, \Sigma[z] \cap \nfR[z]_{2k}.
\]
For a tuple $g=(g_1,\ldots,g_t)$ of polynomials in $z$,
its quadratic module is the set
\[
\qmod[g] \,  \coloneqq  \,  \Sigma[z] +  g_1 \cdot \Sigma[z] + \ldots + g_t \cdot  \Sigma[z].
\]
Similarly, we denote the truncation of $\qmod(g)$
\[
\qmod[g]_{2k} \,  \coloneqq  \, \Sigma[z]_{2k} + g_1\cdot \Sigma[z]_{2k-\deg(g_1)}
+\cdots+g_t\cdot\Sigma[z]_{2k-\deg(g_t)}.
\]
The tuple $g$ determines the basic closed semi-algebraic set
\begin{equation}
  \label{polyrep}
\mathcal{S}(g) \,  \coloneqq  \,  \{z \in \nfR^l: g(z) \ge  0  \}.
\end{equation}
For a tuple $h=(h_1,\ldots,h_s)$ of polynomials in $\re[z]$, its real zero set is
\[
\mcal{Z}(h)  \coloneqq  \{u \in\re^l: h_1(u)=\cdots=h_s(u)=0\}.
\]
The set $\idl[h]+\qmod[g]$ is said to be {\it archimedean}
if there exists $\rho \in \idl[h]+\qmod[g]$ such that the set $\mathcal{S}(\rho)$ is compact.
If $\idl[h]+\qmod[g]$ is archimedean, then $\mathcal{Z}(h)\cap\mathcal{S}(g)$ must be compact.
Conversely, if $\mathcal{Z}(h)\cap\mathcal{S}(g)$ is compact, say,
$\mathcal{Z}(h)\cap\mathcal{S}(g)$ is contained in the ball $R -\|z\|^2 \ge 0$, then $\idl(h)+\qmod(g,R -\|z\|^2)$ is archimedean and $\mathcal{Z}(h)\cap\mathcal{S}(g) = \mathcal{Z}(h)\cap\mathcal{S}(g, R -\|z\|^2)$.
Clearly, if $f \in \idl[h]+\qmod[g]$, then $f \ge 0$
on $\mathcal{Z}(h) \cap \mathcal{S}(g)$.
The reverse is not necessarily true.
However, when $\idl[h]+\qmod[g]$ is archimedean, if $f > 0$
on $\mathcal{Z}(h)\cap\mathcal{S}(g)$, then $f \in \idl[h]+\qmod[g]$.
This conclusion is referenced as Putinar's Positivestellensatz \cite{putinar1993positive}.
Interestingly, if $f \ge 0$ on $\mathcal{Z}(h)\cap\mathcal{S}(g)$, we also have $f\in \idl[h]+\qmod[g]$, under some standard optimality conditions \cite{nie2014optimality}.

\subsection{Localizing and moment matrices}

Let $\re^{ \N_{2k}^{l} }$ denote the space of all real vectors that are labeled by
$\af \in \N_{2k}^{l}$.
Each $y \in \re^{ \N_{2k}^{l} }$ is labeled as
\[
y \, = \, (y_\af)_{ \af \in \N_{2k}^{l} }.
\]
Such $y$ is called a {\it truncated multi-sequence} (tms) of degree $2k$.
For a polynomial $f  = \sum_{ \af \in \N^l_{2k} } f_\af z^\af \in  \re[z]_{2k}$,
define the operation
\be \label{<f,y>}
\langle f, y \rangle \,= \, \sum_{ \af \in \N^l_{2k} } f_\af y_\af.
\ee
The operation $\langle f, y \rangle$ is a bilinear function in $(f, y)$.
For a polynomial $q \in \re[z]$ with $\deg(q) \le 2k$ and the integer
\[
t \,= \, k - \lceil \deg(q)/2 \rceil ,
\]
the outer product $q \cdot [z]_t ([z]_t)^T$ is a symmetric matrix polynomial in $z$,
with length $\binom{n+t}{t}$.
We write the expansion as
\[
q \cdot [z]_t ([z]_t)^T \, = \, \sum_{ \af \in \N_{2k}^l }
z^\af  Q_\af,
\]
for some symmetric matrices $Q_\af$.
Then we define the matrix function
\be \label{df:Lf[y]}
L_{q}^{(k)}[y] \,  \coloneqq  \, \sum_{ \af \in \N_{2k}^l } y_\af  Q_\af.
\ee
It is called the $k$th {\it localizing matrix} of $q$ and generated by $y$.
For given $q$, $L_{q}^{(d)}[y]$ is linear in $y$.
Clearly, if $q(u) \geq 0$ and $y = [u]_{2k}$, then
\[ L_{q}^{(k)}[y] = q(u) [u]_t[u]_t^T \succeq 0. \]
For instance, if $l=k=2$
and $q(z)= 1 - z_1-z_1z_2$, then
\[
L_q^{(2)}[y]=\left [\begin{matrix}
y_{00}-y_{10}-y_{11} &  y_{10}-y_{20}-y_{21} &  y_{01}-y_{11}-y_{12} \\
y_{10}-y_{20}-y_{21} &  y_{20}-y_{30}-y_{31} &  y_{11}-y_{21}-y_{22} \\
y_{01}-y_{11}-y_{12} &  y_{11}-y_{21}-y_{22} &  y_{02}-y_{12}-y_{13} \\
\end{matrix}\right ].
\]
When $q$ is the constant one polynomial,
the localizing matrix $L_{1}^{(k)}[y]$
reduces to a moment matrix, which we denote as
\be\label{eq:Mdy}
M_k[y] \,  \coloneqq  \, L_{1}^{(k)}[y].
\ee
For instance, for $n=2$ and $y \in \re^{ \N^2_4 }$, we have
$M_0[y] = [y_{00}]$,
\[
M_1[y]= \left[
\begin{array}{cccccc}
y_{00} & y_{10} & y_{01}  \\
y_{10} & y_{20} & y_{11}  \\
y_{01} & y_{11} & y_{02} \\
\end{array} \right],
\quad
M_2[y]= \left[
\begin{array}{cccccc}
y_{00} & y_{10} & y_{01} & y_{20} & y_{11} & y_{02} \\
y_{10} & y_{20} & y_{11} & y_{30} & y_{21} & y_{12} \\
y_{01} & y_{11} & y_{02} & y_{21} & y_{12} & y_{03} \\
y_{20} & y_{30} & y_{21} & y_{40} & y_{31} & y_{22} \\
y_{11} & y_{21} & y_{12} & y_{31} & y_{22} & y_{13} \\
y_{02} & y_{12} & y_{03} & y_{22} & y_{13} & y_{04} \\
\end{array} \right].
\]
Localizing and moment matrices are basic tools to
formulate semidefinite relaxations for polynomial optimization problems.
They are important tools for solving
polynomial, matrix, and tensor optimization problems
\cite{HilNie08,PMI2011,Nie2012,Nuc2017,NieZhang18}.

\subsection{Optimality conditions for NEPs}

Consider the $i$th player's individual optimization problem
$\mbf{F}_i(x_{-i})$ in (\ref{eq:subprob}), for given $\xmi$.
Suppose $\mcal{E}_i\cup\mcal{I}_i=[m_i]$ for some $m_i\in\N$.
For convenience, we write the constraining functions as
\[
g_i(x_i) \,  \coloneqq  \, (g_{i,1}(x_i), \ldots, g_{i,m_i}(x_i)).
\]
Suppose $x= (x_1, \ldots, x_N)$ is an NE.
Under linear independence constraint qualification condition (LICQC) at $x_i$,
i.e., the set of gradients for active constraining functions are linearly independent,
there exist Lagrange multipliers $\lambda_{i,j}$ such that
\be
\label{eq:KKTwithLM}
\left\{
\begin{array}{c}
\sum_{j=1}^{m_i}\lambda_{i,j}\nabla_{x_i} g_{i,j}(\xpi) = \nabla_{x_i} f_i(\allx), \\
 0\leq  \lambda_{i,j}\perp g_{i,j}(\xpi) \ge 0  \ (j \in \mcal{I}_i ).
\end{array}
\right.
\ee
In the above, $\lambda_{i,j}\perp g_{i,j}(\xpi)$ means that
$\lambda_{i,j}\cdot g_{i,j}(\xpi)=0$.
This is called the KKT condition for the optimization $\mbf{F}_i(x_{-i})$.
We say a point $x\in\re^n$ is a {\it KKT point} if there exist vectors
of Lagrange multipliers $\lmd_1\ddd \lmd_N$ such that (\ref{eq:KKTwithLM}) holds.
For the NE $x$, if the LICQC of $\mbf{F}_i(x_{-i})$ holds at $\xpi$
for every $i\in[N]$, then $x$ must be a KKT point.
Moreover, if each player's optimization problem is convex, i.e.,
the $f_i(\xpi,\xmi)$ is convex in $\xpi$ for all
$\xmi\in X_1\times \dots\times X_{i-1}\times X_{i+1}\times\dots\times X_N$,
and every $X_i$ is a convex set,
then all KKT points are NEs \cite[Theorem 4.6]{FacKan10OR}.

\begin{example} \rm
Consider the $2$-player NEP in Example~\ref{ep:isolatedep}.
Each individual optimization is strictly convex, because Hessian matrices
$\nabla^2_{x_1}f_1$ and $\nabla^2_{x_2}f_2$ are positive definite.
The constraints are the convex ball conditions. The KKT system is
\be
\label{eq:1.1}
\left\{
\begin{array}{l}
 2x_{1,1}+x_{2,1}+4x_{2,2}=-2\lambda_1 x_{1,1},4x_{1,2}=-2\lambda_1 x_{1,2},\\
 x_{1,1}+2x_{1,2}+2x_{2,1}=-2\lambda_2 x_{2,1},2x_{1,1}+x_{1,2}+2x_{2,2}=-2\lambda_2 x_{2,2},\\
 \lambda_1(1-(x_{1,1})^2-(x_{1,2})^2)=0,\lambda_2(1-(x_{2,1})^2-(x_{2,2})^2)=0,\\
 1-(x_{1,1})^2-(x_{1,2})^2\ge0,1-(x_{2,1})^2-(x_{2,2})^2\ge0,\\
 \lambda_1\ge0, \lambda_2\ge 0.
\end{array}\right.
\ee
By solving the above directly, one can show that
this NEP has only $3$ KKT points, together with Lagrange multipliers as follows
\[
\boxed{
\begin{array}{ll|l}
\multicolumn{2}{c|}{\mbox{Nash equilibrium}} & \multicolumn{1}{c}{\mbox{Lagrange multiplier}}\\ \hline
x_1^* = (0, 0), & x_2^* = (0, 0), &\quad \lambda_1^*=\lambda_2^*=0; \\\hline
x_1^* = (1, 0), & x_2^* = \frac{1}{\sqrt{5}} (-1, -2 ), &\quad \lambda_1^*=\frac{9\sqrt{5}}{10}-1, \,\, \lambda_2^*=\frac{\sqrt{5}}{2}-1;\\\hline
x_1^* = (-1, 0), & x_2^* = \frac{1}{\sqrt{5}}(1, 2 ), &
  \quad \lambda_1^*=\frac{9\sqrt{5}}{10}-1,\,\, \lambda_2^*=\frac{\sqrt{5}}{2}-1.
\end{array}}
\]
All these KKT points are NEs since the NEP is convex.
Furthermore, since for each $i=1,2$, the LICQC of $\mathbf{F}_i(\xmi)$ holds for all $x\in X$,
these NEs are all solutions to the NEP.
This is very different from a single convex optimization problem,
where the set of minimizers, if it is nonempty,
must be a singleton or have an infinite cardinality if the objective function is convex,
and the minimizer has to be unique if the objective function
is further assumed to be strictly convex.
\end{example}

However, the KKT point may not be an NE of the NEP when there is no convexity assumed.
This is because the KKT condition (\ref{eq:KKTwithLM}) is typically not sufficient
for $\xpi$ to be a minimizer of $\mbf{F}_i(x_{-i})$,
which makes nonconvex NEPs quite difficult to solve.
In this paper, we mainly focus on finding NEs for nonconvex NEPs of polynomials.

\section{Polynomial optimization formulations}
\label{sc:finding}
In this section, we show how to formulate efficient polynomial optimization problems
for solving the NEPP \reff{eq:NEP}.
We first introduce the polynomial expressions for Lagrange multiplier expressions in Section~\ref{sc:optlme}.
Then, in Section~\ref{sc:alg},
polynomial optimization problems are formulated for finding NEs,
and an algorithm to solve nonconvex NEPs is proposed.
Convex NEPs of polynomials are studied in Section~\ref{ssc:convex}.
Last, we further extend our approach to find more NEs in Section~\ref{sc: all}.

\subsection{Optimality conditions and Lagrange multiplier expressions}
\label{sc:optlme}

For the NEP (\ref{eq:NEP}), if $x$ is an NE where the LICQC is satisfied,
then it must be a KKT point, i.e., $x$ satisfies (\ref{eq:KKTwithLM}) for all $i\in[N]$.
Therefore, every NE must satisfy the following equation system:
\be
\underbrace{\bbm
\nabla_{x_i} g_{i,1}(\xpi) & \nabla_{x_i} g_{i,2}(\xpi)  &  \cdots &  \nabla_{x_i} g_{i,m_i}(\xpi)  \\
g_{i,1}(\xpi) & 0  & \cdots & 0 \\
0  & g_{i,2}(\xpi)  & \cdots & 0 \\
\vdots & \vdots & \ddots & \vdots \\
0  &  0  & \cdots & g_{i,m_i}(\xpi)
\ebm}_{G_i(x_i) }
\underbrace{
\bbm  \lmd_{i,1} \\ \lmd_{i,2} \\ \vdots \\ \lmd_{i,m_i} \ebm
}_{\lmd_i}
=
\underbrace{\bbm  \nabla_{x_i} f_i(x)  \\ 0 \\ \vdots \\ 0 \ebm}_{ \hat{f}_i(\allx)} .
\ee
If there exists a matrix polynomial $H_i(x_i)$ such that
\be \label{HiGi=Isi}
H_i(x_i) G_i(x_i)  \, = \, I_{m_i},
\ee
then we can express $\lmd_i$ as
\[
\lmd_i = H_i(x_i) G_i(x_i) \lmd_i = H_i(x_i) \hat{f}_i(\allx).
\]
Interestingly, the matrix polynomial $H_i(x_i)$ satisfying \reff{HiGi=Isi}
exists under the nonsingularity condition on $g_i$.
The polynomial tuple $g_i$ is said to be \textit{nonsingular} if
$G_i(x_i)$ has full column rank for all $x_i \in \cpx^{n_i}$ \cite{Nie2018}.
It is a generic condition \cite[Proposition~2.1]{nie2009}.
We remark that if $g_i$ is nonsingular,
then the LICQC holds at every minimizer of \reff{eq:subprob},
so there must exist $\lambda_{i,j}$ satisfying \reff{eq:KKTwithLM}
and we can express $\lmd_{i,j}$ as
\be \label{polyexpr:lmdij}
\lmd_{i,j}  =  \lmd_{i,j}(\allx)  \coloneqq
\big( H_i(x_i) \hat{f}_i(\allx) \big)_j
\ee
for all NEs.
For example, we consider the following two cases:
\bit

\item For the constraint $\{x_i\in\re^{n_i}:\sum_{j=1}^{n_i}x_{i,j}\le 1,x_i\ge0\}$,
the constraining polynomials are
\[
g_{i,0}=1-\sum_{j=1}^{n_i}x_{i,j}, \, g_{i,1}=x_{i,1},\, \ldots,\, g_{i,n_i}=x_{i,n_i}.
\]
If we let
\[H_i(x_i)=\left[\begin{array}{ccccccc}
1-x_{i,1} & -x_{i,2} & \ldots & -x_{i,n_i} & 1 & \ldots & 1\\
-x_{i,1} & 1-x_{i,2} & \ldots & -x_{i,n_i} & 1 & \ldots & 1\\
\vdots  & \vdots  & \ddots  & \vdots  &\vdots   &\vdots &\vdots\\
-x_{i,1} & -x_{i,2} & \ldots & 1-x_{i,n_i} & 1 & \ldots & 1\\
-x_{i,1} &  -x_{i,2} & \ldots & -x_{i,n_i} & 1 & \ldots & 1
\end{array}\right],\]
then one may check that the (\ref{HiGi=Isi}) holds.
The Lagrange multipliers $\lmd_{i,j}$ can be accordingly represented as
\be\label{eq:simplme}
\lmd_{i,0} \,=\, x_i^T\nabla_{x_i} f_i, \quad
\lambda_{i,j} \, = \, \frac{\partial f_i}{\partial x_{i,j}}-x_i^T\nabla_{x_i} f_i ,
\quad j = 1,\ldots, n_i.
\ee

\item For the sphere constraint $1-x_i^Tx_i = 0$
or the ball constraint $1-x_i^Tx_i \ge 0 $,
the constraining polynomial is $g_{i,1}=1-x_i^Tx_i$.
If we let
\[H_i(x_i)=\left[\begin{array}{ccccccc}
-\frac{1}{2}x_{i,1} & -\frac{1}{2}x_{i,2} & \ldots & -\frac{1}{2}x_{i,n_i} & 1
\end{array}\right],\]
then one may check that the (\ref{HiGi=Isi}) holds.
The Lagrange multiplier can be accordingly expressed as
\be\label{eq:balllme}\lambda_{i,1} \, = \, -\frac{1}{2}x_i^T\nabla_{x_i} f_i.\ee

\eit
For general nonsingular constraining tuple,
one may find $H_i(x_i)$ satisfying (\ref{HiGi=Isi}) by solving linear equations.
We refer to \cite{Nie2018} for more details on
getting the polynomial expressions of Lagrange multipliers.

Throughout the paper, we assume that every constraining polynomial tuple
$g_i$ is nonsingular. This is a generic assumption.
So all $\lmd_{i,j}$ can be expressed as polynomials
as in \reff{polyexpr:lmdij}. Then, each Nash equilibrium
satisfies the following polynomial system
\be
\label{eq:criticalpt}
\left\{
\begin{array}{l}
\nabla_{x_i} f_i(\allx)
-\sum_{j=1}^{m_i}\lambda_{i,j}(\allx) \nabla_{x_i} g_{i,j}(x_i) =0 \ (i\in[N]), \\
  g_{i,j}(x_i) = 0 \ (i\in[N], j \in  \mcal{E}_i), \,
\lmd_{i,j}(\allx)  g_{i,j}(x_i) =  0 \ (i\in[N], j \in  \mcal{I}_i), \\
  g_{i,j}(x_i) \geq  0 \ (j \in  \mcal{I}_i), \,
           \lambda_{i,j} (\allx) \geq 0 \ (i\in[N], j \in  \mcal{I}_i) .
\end{array}
\right.
\ee

\subsection{An algorithm for finding an NE}
\label{sc:alg}

For the NEP of polynomials (\ref{eq:NEP}),
let $\lmd_{i,j}(x)$ be polynomial Lagrange multiplier expressions
as in (\ref{polyexpr:lmdij}) for each $i\in[N]$ and $j\in[m_i]$.
Then every NE must satisfy the polynomial system \reff{eq:criticalpt}.
Choose a generic positive definite matrix
\[
\Theta \, \in \, \re^{ (n+1) \times  (n+1)   } .
\]
Then all NEs are feasible points for the following optimization problem
\be
\label{eq:KKTfeasopt}
\left\{
\begin{array}{rll}
\min\limits_{\allx} & [\allx]_1^T \cdot \Theta \cdot [\allx]_1 \\
\st  &  \nabla_{x_i} f_i(\allx)
-\sum_{j=1}^{m_i}\lambda_{i,j}(\allx) \nabla_{x_i} g_{i,j}(x_i) =0 \ (i \in [N]),  \\
  &  g_{i,j}(x_i) = 0 \ (j \in  \mcal{E}_i, i \in [N]), \\
  & \lmd_{i,j}(\allx)  g_{i,j}(x_i) =  0 \ (j \in  \mcal{I}_i, i \in [N]), \\
  &  g_{i,j}(x_i) \geq  0 \ (j \in  \mcal{I}_i, i \in [N]), \\
  & \lambda_{i,j} (\allx) \geq 0 \ (j \in  \mcal{I}_i, i \in [N]) .
\end{array}
\right.
\ee
In the above, the vector $[x]_1  \coloneqq  (1, x_1, x_2, \dots, x_n)^T \in \re^{n+1}$.
Note that $x\in\re^n$ is a KKT point for the NEP
if and only if it is feasible for (\ref{eq:KKTfeasopt}).
It is important to observe that if \reff{eq:KKTfeasopt} is infeasible,
then there are no NEs.
If \reff{eq:KKTfeasopt} is feasible,
then it must have a minimizer,
because its objective is a positive definite quadratic function.
Moreover, for a generic $\Theta\in\re^{(n+1)\times (n+1)}$,
the minimizer of \reff{eq:KKTfeasopt} is unique (see Theorem~\ref{thmcvg:upconv}).
The (\ref{eq:KKTfeasopt}) is a polynomial optimization problem,
which can be solved by the Moment-SOS semidefinite relaxations
(see Section~\ref{sc:solvepop}).

Assume that $u  \coloneqq  (u_1, \ldots, u_N)$
is an optimizer of (\ref{eq:KKTfeasopt}).
Then $u$ is an NE if and only if each $u_i$
is a minimizer of $\mathbf{F}_i(u_{-i})$.
To this end, for each player,
consider the optimization problem:
\be  \label{eq:lowerlvopt}
\left\{
\begin{array}{rcl}
\omega_i\, \coloneqq \, & \min & f_i(x_i, u_{-i})-f_i(u_i, u_{-i}) \\
        &    \st  &  g_{i,j}(\xpi) = 0 \ (j \in \mcal{E}_i),  \\
        &         &  g_{i,j}(\xpi) \ge 0 \ (j \in \mcal{I}_i) .
\end{array}
\right.
\ee
If all the optimal values $\omega_i \ge 0$, then $u$ is a Nash equilibrium.
If one of them is negative, say, $\omega_i < 0$, then $u$ is not an NE.
For such a case, let $U_i$ be a set of some optimizers of \reff{eq:lowerlvopt},
then $u$ violates the following inequalities
\be \label{cutineq:NE}
f_i(x_i, x_{-i}) \le f_i(v, x_{-i}) \quad  (v \in U_i) .
\ee
However, every Nash equilibrium must satisfy \reff{cutineq:NE}.

When $u$ is not an NE, we aim at finding a new candidate
by posing the inequalities in \reff{cutineq:NE}.
Therefore, we consider the following optimization problem:
\be  \label{eq:upperlvopt}
\left\{
  \begin{array}{rl}
    \min\limits_{\allx} &  [\allx]_1^T \cdot \Theta \cdot [\allx]_1   \\
    \st  &  \nabla_{x_i} f_i(\allx)
-\sum_{j=1}^{m_i}\lambda_{i,j}(\allx) \nabla_{x_i} g_{i,j}(x_i) =0 \ (i \in [N]),  \\
  &  g_{i,j}(x_i) = 0 \ (j \in  \mcal{E}_i, i \in [N]), \\
  & \lmd_{i,j}(\allx)  g_{i,j}(x_i) =  0 \ (j \in  \mcal{I}_i, i \in [N]), \\
  &  g_{i,j}(x_i) \geq  0 \ (j \in  \mcal{I}_i, i \in [N]), \\
  & \lambda_{i,j} (\allx) \geq 0 \ (j \in  \mcal{I}_i, i \in [N]),  \\
 & f_i(v, x_{-i}) - f_i(x_i, x_{-i}) \ge 0\ (v \in \mcal{K}_i, i\in[N]) .
  \end{array}
  \right.
\ee
In the above, each $\mcal{K}_i$ is a set of some optimizers of
\reff{eq:lowerlvopt}.
We solve \reff{eq:upperlvopt} again for a minimizer, say, $\hat{u}$.
If $\hat{u}$ is verified to be an NE, then we are done.
If it is not, we can add more inequalities like \reff{cutineq:NE}
to exclude both $u$ and $\hat{u}$.
Repeating this procedure, we get the following algorithm for computing an NE.

\begin{alg}
\label{ag:KKTSDP}
For the NEP given as in \reff{eq:subprob} and \reff{eq:NEP}, do the following

\begin{itemize}

\item [Step~0]
Initialize $\mcal{K}_i  \coloneqq  \emptyset$ for all $i$ and $\ell  \coloneqq  0$.
Choose a generic positive definite matrix $\Theta$ of length $n+1$.

\item [Step~1]
Solve the polynomial optimization problem \reff{eq:upperlvopt}.
If it is infeasible, then output that there is no NE and stop; otherwise,
solve it for an optimizer $u$.

\item [Step~2]
For each $i=1,\ldots,N$, solve the optimization \reff{eq:lowerlvopt}.
If all $\omega_i \ge 0$, then output the NE $u$ and stop.
If one of $\omega_i$ is negative, then go to the next step.

\item [Step~3]
For each $i$ with $\omega_i < 0$,
obtain a set $U_i$ of some (may not all) optimizers of \reff{eq:lowerlvopt};
then update the set $\mcal{K}_i  \coloneqq  \mcal{K}_i \cup U_i$.
Let $\ell  \coloneqq  \ell+1$, then go to Step~1.

\end{itemize}
\end{alg}

In the Step~0, we can set $\Theta = R^T R$
for a randomly generated matrix $R$ of length $n+1$.
The objective in \reff{eq:upperlvopt} is a positive definite quadratic function,
so it has a minimizer if \reff{eq:upperlvopt} is feasible.
The case is slightly different for (\ref{eq:lowerlvopt}).
If the feasible set $X_i$ is compact or $f_i(x_i, u_{-i})$ is coercive
for the given $u_{-i}$, then (\ref{eq:lowerlvopt}) has a minimizer.
If $X_i$ is unbounded and $f_i(x_i, u_{-i})$ is not coercive,
it may be difficult to compute the optimal value $\omega_i$.
In applications, we are mostly interested in cases that (\ref{eq:lowerlvopt})
has a minimizer, for the existence of an NE.
We discuss how to solve the optimization problems in Algorithm~\ref{ag:KKTSDP}
by the Moment-SOS hierarchy of semidefinite relaxations in Section~\ref{sc:solvepop}.

The following is the convergence theorem for Algorithm~\ref{ag:KKTSDP}.
\begin{theorem}  \label{tm:finiteconv}
Assume each constraining polynomial tuple $g_i$ is nonsingular
and let $\lmd_{i,j}(\allx)$ be polynomial expressions of Lagrange multipliers as in \reff{polyexpr:lmdij}.
Let $\mcal{G}$ be the feasible set of \reff{eq:KKTfeasopt} and $\mcal{G}^*$
be the set of all NEs.
If the complement $\mcal{G} \backslash \mcal{G}^*$ is a finite set, i.e.,
the cardinality $\ell^*  \coloneqq  |\mcal{G} \backslash \mcal{G}^*| < \infty$,
then Algorithm~\ref{ag:KKTSDP} must terminate within at most $\ell^*$ loops.
\end{theorem}
\begin{proof}
Under the nonsingularity assumption of polynomial tuples $g_i$,
the Lagrange multipliers $\lmd_{i,j}$ can be expressed as polynomials
$\lmd_{i,j}(\allx)$ as in \reff{polyexpr:lmdij}.
For each $u$ that is a feasible point of \reff{eq:KKTfeasopt},
every NE must satisfy the constraint
\[
f_i(u_i, x_{-i}) - f_i(x_i, x_{-i}) \ge 0.
\]
Therefore, every NE must also be a feasible point of \reff{eq:upperlvopt}.
Since the matrix $\Theta$ is positive definite,
the optimization \reff{eq:upperlvopt} must have a minimizer,
unless it is infeasible.
When Algorithm~\ref{ag:KKTSDP} goes to a newer loop,
say, from the $\ell$th to the $(\ell+1)$th,
the optimizer $u$ for \reff{eq:upperlvopt} in the $\ell$th loop
is no longer feasible for \reff{eq:upperlvopt} in the $(\ell+1)$th loop.
This means that the feasible set of \reff{eq:upperlvopt}
must loose at least one point after each loop,
unless an NE is met.
Also note that the feasible set of \reff{eq:upperlvopt}
is contained in $\mcal{G}$. If $\mcal{G} \backslash \mcal{G}^*$
is a finite set, Algorithm \ref{ag:KKTSDP}
must terminate after some loops.
The number of loops is at most $\ell^*$.
\end{proof}

As shown in the appendix, when the NEP is given by generic polynomials,
the NEP has finitely many KKT points (see Theorem~\ref{tm:genfinite}).
For such cases, $|\mcal{G} \backslash \mcal{G}^*|\le |\mcal{G}|<\infty$
and finite termination of Algorithm~\ref{ag:KKTSDP} is guaranteed.

\begin{theorem} \label{tm:genfiniteconv}
Let $d_{i,j}>0$, $a_{i,j} >0$ be degrees, for all $i \in [N], j \in  [m_i]$.
If each $g_{i,j}$ is a generic polynomial in $x_i$ of degree $d_{i,j}$
and each $f_i$ is a generic polynomial in $\allx$ whose degree in $x_j$ is $a_{i,j}$,
then Algorithm~\ref{ag:KKTSDP} terminates within finitely many loops, i.e.,
it either finds an NE if there exists any, or detect nonexistence of NEs.
\end{theorem}
\begin{proof}
The conclusion follows directly from
Theorems~\ref{tm:finiteconv} and \ref{tm:genfinite}.
\end{proof}

When there exist infinitely many KKT points that are not NEs,
Algorithm~\ref{ag:KKTSDP} can still be applied to compute an NE
if there exists one, or detect nonexistence of NEs if they do not exist.
See Example~\ref{ep:infiKKT}(ii) for such a case.
However, for such NEPPs, the convergence property of
Algorithm~\ref{ag:KKTSDP} is not fully understood.

\subsection{Convex NEPs}
\label{ssc:convex}

The NEP is said to be {\it convex} if for every $i\in[N]$,
the $f_i(\xpi,\xmi)$ is convex in $\xpi$ for all $\xmi\in X_{-i} \coloneqq \prod_{j\in[N]\setminus\{i\}}X_j$, the $g_{i,j}(\xpi)$
is linear for each $j\in\mcal{E}_i$, and is concave for every $j\in\mcal{I}_i$.
For convex NEPs, every KKT point must be an NE,
since the KKT conditions are sufficient for global optimality.

Moreover, for convex NEPPs, when every constraining tuple $g_i$ is nonsingular,
the LICQC holds for all $x\in X$, and a point is an NE
if and only if it satisfies the KKT conditions.
Note that the Lagrange multipliers can be expressed
by polynomials as in (\ref{polyexpr:lmdij})
when nonsingularity is assumed.
For such cases, the solution set for (\ref{eq:KKTwithLM})
is exactly the set of NEs.
Therefore, if we solve the polynomial optimization problem (\ref{eq:upperlvopt}) with $\mcal{K}_i=\emptyset$ for all $i\in[N]$
(i.e., the polynomial optimization (\ref{eq:KKTfeasopt})),
then every minimizer, if the feasible set is nonempty, must be an NE.
On the other hand, if (\ref{eq:upperlvopt}) is infeasible,
then we immediately know the NEs do not exist.
This shows that, for convex NEPPs, Algorithm~\ref{ag:KKTSDP}
must terminate at the initial loop.

\begin{corollary} \label{cor:cvxnep}
Assume each $g_i$ is a nonsingular tuple of polynomials.
Suppose each $g_{i,j}(x_i)$ ($j \in \mcal{E}_i$) is linear,
each $g_{i,j}(x_i)$ ($j \in \mcal{I}_i$) is concave,
and each $f_i(x_i,\xmi)$ is convex in $x_i$ for all $x_{-i}\in X_{-i}$.
Then Algorithm~\ref{ag:KKTSDP} must terminate at the first loop with $\ell = 0$,
returning an NE or reporting that there is no NE.
\end{corollary}

\begin{example}\rm\label{ep:cvxcomp}
Consider the convex NEP in Example~\ref{ep:isolatedep}.
In this NEP,
both players have ball constraints,
so their Lagrange multipliers can be expressed by polynomials as in (\ref{eq:balllme}).
We ran Algorithm~\ref{ag:KKTSDP}
for solving this NEP\footnote{See Section~\ref{sc:solvepop}
for how to solve polynomial optimization problems, and Section~\ref{sc:num}
for computational information.},
and found the NE $x^*=(x_1^*, x_2^*)$ with
\[x_1^*=(-1.0000, 0.0000),\quad x_2^*=(0.4472, 0.8944)\]
in the initial loop.
It took around 0.88 second.
\end{example}

\subsection{More Nash equilibria}
\label{sc: all}

Algorithm~\ref{ag:KKTSDP} aims at finding a single NE.
In some applications, people may be interested in more NEs.
Moreover, when there is a unique NE, people are also interested
in a certificate for uniqueness.

In this subsection, we study how to find more NEs or
check the completeness of solution sets.
Assume that $x^*$ is a Nash equilibrium produced by Algorithm~\ref{ag:KKTSDP},
i.e., $x^*$ is also a minimizer of \reff{eq:upperlvopt}.
Then all KKT points $x$ satisfying $[x]_1^T \Theta [x]_1 \, < \,
[x^*]_1^T \Theta [x^*]_1$
are excluded from the feasible set of \reff{eq:upperlvopt} by the constraints
\[
f_i(u_i,\xmi)-f_i(x_i,\xmi) \,\ge \, 0  \quad
(\forall \, u \in \mcal{K}_i,\, \forall \, i\in[N]).
\]
If $x^*$ is an isolated NE (e.g., this is the case if there are finitely many NEs),
there exists a scalar $\dt >0$ such that
\be\label{eq:sep}
[x]_1^T \Theta [x]_1 \, \ge \,
[x^*]_1^T \Theta [x^*]_1 + \dt
\ee
for all other NEs $\allx$.
For such a $\dt$, we can try to find a different NE by
solving the following optimization problem
\be
\label{eq:upperlvoptall}
\left\{ \begin{array}{cl}
\min\limits_{\allx} & [\allx]_1^T \Theta [\allx]_1 \\
\st  &  \nabla_{x_i} f_i(\allx)
-\sum_{j=1}^{m_i}\lambda_{i,j}(\allx) \nabla_{x_i} g_{i,j}(x_i) =0 \ (i \in [N]),  \\
  &  g_{i,j}(x_i) = 0 \ (j \in  \mcal{E}_i, i \in [N]), \\
  & \lmd_{i,j}(\allx)  g_{i,j}(x_i) =  0 \ (j \in  \mcal{I}_i, i \in [N]), \\
  &  g_{i,j}(x_i) \geq  0 \ (j \in  \mcal{I}_i, i \in [N]), \\
  & \lambda_{i,j} (\allx) \geq 0 \ (j \in  \mcal{I}_i, i \in [N]),  \\
  & f_i(v, x_{-i}) - f_i(x_i, x_{-i}) \ge 0\ (v \in \mcal{K}_i, i\in[N]),  \\
  & [\allx]_1^T \Theta [\allx]_1 \, \ge \, [\allx^*]_1^T \Theta [\allx^*]_1 + \dt .
\end{array} \right.
\ee
When an optimizer of (\ref{eq:upperlvoptall}) is computed,
we can check if it is an NE or not by solving (\ref{eq:lowerlvopt}) for all $i\in[N]$.
If it is, we get a new NE that is different from $\allx^*$.
If it is not, we update the set $\mcal{K}_i$ as in Step~3 of Algorithm~\ref{ag:KKTSDP}.
Repeating the above process, we are able to get more Nash equilibria.

A concern in computation is how to choose the constant
$\dt >0$ for \reff{eq:upperlvoptall}.
We want a value $\dt >0$ such that (\ref{eq:sep}) holds for all unknown NEs.
To this end, we consider the following maximization problem
\be \label{eq:maxxi}
\left\{
\begin{array}{cl}
\max\limits_{x} & [x]_1^T \Theta [x]_1   \\
\st  &  \nabla_{x_i} f_i(x)
-\sum_{j=1}^{m_i}\lambda_{i,j}(\allx) \nabla_{x_i} g_{i,j}(x_i) =0 \ (i \in [N]),  \\
  &  g_{i,j}(x_i) = 0 \ (j \in  \mcal{E}_i, i \in [N]), \\
  & \lmd_{i,j}(x)  g_{i,j}(x_i) =  0 \ (j \in  \mcal{I}_i, i \in [N]), \\
  &  g_{i,j}(x_i) \geq  0 \ (j \in  \mcal{I}_i, i \in [N]), \\
  & \lambda_{i,j} (x) \geq 0 \ (j \in  \mcal{I}_i, i \in [N]),  \\
  & f_i(v, x_{-i}) - f_i(x_i, x_{-i}) \ge 0\ (v \in \mcal{K}_i, i\in[N]),  \\
  & [\allx]_1^T \Theta [\allx]_1 \, \le \, [\allx^*]_1^T \Theta [\allx^*]_1 + \dt .\\
\end{array} \right.
\ee
Interestingly, if $x^*$ is also a maximizer of \reff{eq:maxxi},
i.e., the maximum of \reff{eq:maxxi} equals $[x^*]^T\Theta[x^*]_1$,
then the feasible set of \reff{eq:upperlvoptall}
contains all NEs except $\allx^*$,
under some general assumptions.

\begin{prop} \label{prop:choice:dt}
Assume $\Theta$ is a generic positive definite matrix,
and $\allx^*$ is a minimizer of \reff{eq:upperlvopt}.
\bit

\item [(i)]
If $\allx^*$ is also a maximizer of \reff{eq:maxxi},
then there is no other Nash equilibrium $u$
satisfying $[u]_1^T \Theta [u]_1 \, \le \, [\allx^*]_1^T \Theta [\allx^*]_1 + \dt $.

\item [(ii)]
If $\allx^*$ is an isolated KKT point, then there exists $\dt > 0$
such that $\allx^*$ is also a maximizer of \reff{eq:maxxi}.

\eit
\end{prop}
\begin{proof}
Note that every NE is a feasible point of \reff{eq:upperlvopt}.

(i) If $\allx^*$ is also a maximizer of \reff{eq:maxxi},
then the objective $[\allx]_1^T \Theta [\allx]_1$
achieves a constant value in the following set of \reff{eq:maxxi}.
If $u$ is a Nash equilibrium with
$[u]_1^T \Theta [u]_1 \, \le \, [\allx^*]_1^T \Theta [\allx^*]_1 + \dt$,
then
\[
\, [u]_1^T \Theta [u]_1 = [\allx^*]_1^T \Theta [\allx^*]_1 .
\]
This means that $u$ is also a minimizer of \reff{eq:upperlvopt}.
When $\Theta$ is a generic positive definite matrix,
the optimization \reff{eq:upperlvopt} has a unique optimizer,
so $u=\allx^*$.

(ii) Since $\Theta$ is positive definite, there exists $\eps > 0$ such that
\[
 [\allx]_1^T \Theta [\allx]_1  \, \ge \,  \eps( 1 + \| x \|)^2
\]
for all $\allx$. Let $C = \sqrt{ \big( [\allx^*]_1^T \Theta [\allx^*]_1 \big) / \eps }$,
then the following set
\be \nn
T \,  \coloneqq  \,
\left\{ y = [x]_2
\left| \begin{array}{c}
  \nabla_{x_i} f_i(\allx)
-\sum_{j=1}^{m_i}\lambda_{i,j}(\allx) \nabla_{x_i} g_{i,j}(x_i) =0 \ (i \in [N]),  \\
    g_{i,j}(x_i) = 0 \ (j \in  \mcal{E}_i, i \in [N]), \\
   \lmd_{i,j}(\allx)  g_{i,j}(x_i) =  0 \ (j \in  \mcal{I}_i, i \in [N]), \\
     g_{i,j}(x_i) \geq  0 \ (j \in  \mcal{I}_i, i \in [N]), \\
   \lambda_{i,j} (\allx) \geq 0 \ (j \in  \mcal{I}_i, i \in [N]),  \\
   f_i(v, x_{-i}) - f_i(x_i, x_{-i}) \ge 0\ (v \in \mcal{K}_i, i\in[N]),  \\
   \, \| x \| \leq C
\end{array} \right. \right\}
\ee
is compact. Note that $[x^*]_2 \in T$.
Let $\theta$ be the vector such that
\[
 [x]_1^T \Theta [x]_1 \,=\, \theta^T y
\]
for all $y = [x]_2$. Since $\allx^*$ is an isolated KKT point,
the $y^*  \coloneqq  [x^*]_2$ is also an isolated point of $T$.
Then its subset
\[
T_1 \,  \coloneqq  \, T \backslash \{ y^* \}
\]
is also a compact set.
Since $\allx^*$ is a minimizer of \reff{eq:upperlvopt},
the hyperplane $H  \coloneqq \{ \theta^T y = \theta^T y^* \}$
is a supporting hyperplane for the set $T$.
Since $\Theta$ is generic, the optimization~\reff{eq:upperlvopt}
has a unique minimizer, which implies that $y^*$
is the unique minimizer of the linear function $\theta^Ty$ on $T$.
So, $H$ does not intersect $T_1$, and their distance is positive.
There exists a scalar $\tau > 0$ such that
\[
 [\allx]_1^T \Theta [\allx]_1 \,=\, \theta^T y \, \ge \, \theta^T y^* + \tau
  \,=\, [\allx^*]_1^T \Theta [\allx^*]_1 + \tau
\]
for all $y = [x]_2 \in T_1$.
Then, for the choice $\dt  \coloneqq  \tau/2$,
the point $x^*$ is the only feasible point for \reff{eq:maxxi}.
Hence, $\allx^*$ is also a maximizer of \reff{eq:maxxi}.
\end{proof}

Proposition~\ref{prop:choice:dt} shows the existence of $\dt>0$
such that \reff{eq:upperlvopt} and \reff{eq:maxxi}
have the same optimal value.
However, it does not give a concrete lower bound for $\dt$.
In computational practice, we can first give a priori value for $\dt$.
If it does not work, we can decrease $\delta$ to a smaller value
(e.g., let $\delta \coloneqq  \delta/5$).
By repeating this, the optimization (\ref{eq:maxxi})
will eventually have $x^*$ as a maximizer.
The following is the algorithm for finding an NE that is different from $x^*$.

\begin{alg}
\label{ag:allKKTSDP}
For the given NEP (\ref{eq:NEP}) and a computed NE $x^*$,
let $\Theta$ be the positive definite matrix for computing $x^*$.
\begin{itemize}
\item [Step~0] Give an initial value for $\dt$ (say, $0.1$).
\item [Step~1] Solve the maximization problem \reff{eq:maxxi}.
If its optimal value $\eta$ equals
$\upsilon  \coloneqq  [\allx^*]_1^T \Theta [\allx^*]_1$,
then go to Step~2.
If $\eta$ is bigger than $\upsilon$,
then let $\delta \coloneqq \delta/5$
and repeat this step.

\item [Step~2]
Solve the optimization problem \reff{eq:upperlvoptall}.
If it is infeasible, then output there are no additional NEs and stop;
otherwise, solve \reff{eq:upperlvoptall} for a minimizer $u$.

\item [Step~3]
For each $i=1,\ldots,N$, solve the optimization \reff{eq:lowerlvopt}
for the optimal value $\omega_i$.
If all $\omega_i \ge 0$, stop and output the new NE $u$.
If one of $\omega_i$ is negative, then go to Step~4.

\item [Step~4]
For each $i\in[N]$, update the set $\mcal{K}_i \coloneqq \mcal{K}_i \cup U_i$,
and then go back to Step~2.

\end{itemize}
\end{alg}

When $\allx^*$ is not an isolated KKT point, there may not
exist a satisfactory $\dt >0$ for Step~1.
For such a case, more investigation is required to verify
the completeness of the solution set or to find other NEs.
However, for generic NEPs, there are finitely many KKT points
(see Theorem~\ref{tm:genfinite} in the appendix).
The following is the convergence result for Algorithm~\ref{ag:allKKTSDP}.

\begin{theorem}
\label{tm:allnash}
Under the same assumptions in Theorem~\ref{tm:finiteconv},
if $\Theta$ is a generic positive definite matrix and
$\allx^*$ is an isolated KKT point,
then Algorithm~\ref{ag:allKKTSDP} must terminate after finitely many steps,
either returning an NE that is different from $\allx^*$
or reporting the nonexistence of other NEs.
\end{theorem}
\begin{proof}
Under the given assumptions, Proposition~\ref{prop:choice:dt}(ii)
shows the existence of $\dt > 0$ satisfactory
for the Step~1 of Algorithm~\ref{ag:allKKTSDP}.
Again, by Proposition~\ref{prop:choice:dt}(i),
the feasible set of \reff{eq:upperlvoptall} contains all NEs
except $\allx^*$. The finite termination of Algorithm~\ref{ag:allKKTSDP}
can be proved in the same way as for Theorem~\ref{tm:finiteconv}.
\end{proof}

Once a new NE is obtained, we can repeatedly apply Algorithm~\ref{ag:allKKTSDP},
to compute more NEs, if they exist. In particular, if there are finitely many NEs,
then we enumerate them as
\[
\left(\allx^{(1)}, \, \ldots, \, \allx^{(s)}\right).
\]
Without loss of generality, we assume
\[
[\allx^{(1)}]_1^T \Theta [\allx^{(1)}]_1 < \cdots
< [\allx^{(s)}]_1^T \Theta [\allx^{(s)}]_1 ,
\]
since $\Theta$ is generic. If the first $r$ NEs, say,
$\allx^{(1)}, \, \ldots, \, \allx^{(r)}$, are obtained,
there exists $\delta>0$ such that
\[
[\allx^{(j)}]_1^T \Theta [\allx^{(j)}]_1 >
    [\allx^{(r)}]_1^T \Theta [\allx^{(r)}]_1+\delta
\]
for all $j=r+1,\ldots,s$. Therefore, if we apply Algorithm~\ref{ag:allKKTSDP}
with $\allx^*=\allx^{(r)}$, the next Nash equilibrium $\allx^{(r+1)}$
can be obtained, if it exists.
Therefore, we have the following conclusion.

\begin{corollary} \label{cor:allNE}
Under the assumptions of Theorem~\ref{tm:allnash},
if there are finitely many Nash equilibria,
then all of them can be found by applying Algorithm~\ref{ag:allKKTSDP} repeatedly.
\end{corollary}
\begin{remark}
Under the assumption of Theorem~\ref{tm:genfiniteconv},
the NEP has finitely many KKT points.
For such cases, Algorithm~\ref{ag:allKKTSDP} can find all NEs
and certify the completeness of solutions set within finitely many steps, by Corollary~\ref{cor:allNE}.
\end{remark}

\section{Solve polynomial optimization problems}
\label{sc:solvepop}

In this section, we discuss how to solve occurring polynomial optimization problems in Algorithms~\ref{ag:KKTSDP} and \ref{ag:allKKTSDP}.
For the NEP, we assume the constraining polynomial tuples
$g_i$ are all nonsingular.
Therefore, the Lagrange multipliers $\lmd_{i,j}$
can be expressed as polynomial functions $\lmd_{i,j}(\allx)$
as in \reff{polyexpr:lmdij} for all Nash equilibria.
We apply the Moment-SOS hierarchy of
semidefinite relaxations \cite{HKL20,Las01,LasBk15,Lau09}
for solving these polynomial optimization problems.
New convergence results for solving these polynomial optimization problems
are given due to the usage of polynomial expressions for Lagrange multipliers.

For the variable $z$ such that $z=x$ or $z=x_i$ for some $i\in[N]$,
denote by $l$ the dimension of $z$.
Consider the polynomial optimization problem in the variable $z$:
\be  \label{eq:genopt}
\left\{
\baray{ccl}
\vartheta^* \, \coloneqq \, &  \min\limits_{z\in\re^l} &
     \theta(z)   \\
&\st &  p(z) = 0 \  (\forall \, p \in \Phi),   \\
&    &  q(z) \ge 0 \  (\forall \, q \in \Psi).
\earay
\right.
\ee
In the above, $\Phi$ and $\Psi$ are sets of equality
and inequality constraining polynomials, respectively.
Denote the degree
\be\label{eq:d0}
d_0 \,  \coloneqq  \, \max\{\lceil\deg(p)/2\rceil: \,
 p \in \{\theta\}\cup \Phi \cup \Psi \}.
\ee
For a degree $k\ge d_0$, recall that the set $\idl[\Phi]_{2k}+\qmod[\Psi]_{2k}$
is introduced in Section~\ref{sc:idqm}.
The $k$th order SOS relaxation for (\ref{eq:genopt}) is
\be
\label{eq:d-sos}
\left\{
\baray{ccl}
\vartheta_{sos}^{(k)} \, \coloneqq \, & \max & \gamma\\
& \st & \theta -\gamma \in  \idl[\Phi]_{2k}+\qmod[\Psi]_{2k}. \\
\earay
\right.
\ee
The dual problem of (\ref{eq:d-sos}) is the $k$th order moment relaxation
\be
\label{eq:d-mom}
\left\{
\baray{ccl}
\vartheta^{(k)}_{mom}  \,  \coloneqq   \, & \min\limits_{ y } & \lip \theta,y \rip\\
 &\st & y_0=1,\, L_{p}^{(k)}[y] = 0 \ (p \in \Phi), \\
 & & M_d[y] \succeq 0,\, L_{q}^{(k)}[y] \succeq 0 \ (q \in \Psi),  \\
 & &  y \in \mathbb{R}^{\mathbb{N}^{l}_{2k}},
\earay
\right.
\ee
where the moment matrix $M_k[y]$ and localizing matrices
$L_{p}^{(k)}[y],\, L_{q}^{(k)}[y]$ are given by (\ref{df:Lf[y]}) and (\ref{eq:Mdy}).
Both (\ref{eq:d-sos}) and (\ref{eq:d-mom}) are semidefinite programs,
and the primal-dual pair is called the Moment-SOS relaxations
for the polynomial optimization problem (\ref{eq:genopt}).
If $z\in\re^l$ is a feasible point of (\ref{eq:genopt}),
then $[z]_k\in\re^l_{2k}$ must be a feasible point of (\ref{eq:d-mom}).
Thus (\ref{eq:genopt}) has an empty feasible set if (\ref{eq:d-mom}) is infeasible.
When (\ref{eq:d-mom}) has a nonempty feasible set,
it is clear that $\vartheta^{(k)}_{mom}\le \vartheta^{(k)}_{sos}\le \vartheta^*$
for all $k$, and both $\vartheta^{(k)}_{mom}$
and $\vartheta^{(k)}_{sos}$ are monotonically increasing.
The following is the Moment-SOS algorithm for solving \reff{eq:genopt}.

\begin{alg} \label{ag:momsos}
For the polynomial optimization problem (\ref{eq:genopt}),
let $d_0$ be the degree given by (\ref{eq:d0}).
\begin{itemize}

\item [Step~0]
Initialize $k \coloneqq d_0$.

\item [Step~1]
Solve the moment relaxation (\ref{eq:d-mom}).
If it is infeasible, then the polynomial optimization problem (\ref{eq:genopt})
is infeasible and stop;
otherwise, solve (\ref{eq:d-mom}) for the minimum value $\vartheta^{(k)}_{mom}$
and a minimizer $y^{(k)}$.

\item [Step~2]
Let $t \coloneqq d_0$.
If $y^*$ satisfies the rank condition
\be \label{eq:flatrank}
 \Rank{M_t[y^*]} \,=\, \Rank{M_{t-d_0}[y^*]} ,
\ee
then extract a set $U_i$ of
 $r  \coloneqq \Rank{M_t(y^*)}$ minimizers for (\ref{eq:genopt}) and stop.

\item [Step~3]
If \reff{eq:flatrank} fails to hold and $t < k$,
let $t  \coloneqq  t+1$ and then go to Step~2;
otherwise, let $k  \coloneqq  k+1$ and go to Step~1.

\end{itemize}
\end{alg}

Algorithm~\ref{ag:momsos} is known as the Moment-SOS hierarchy of
semidefinite relaxations \cite{Las01}.
We say the Moment-SOS hierarchy has asymptotic convergence if
$\vartheta^{(k)}_{sos}\to \vartheta^*$ as $k\to \infty$,
and we say it has finite convergence if $\vartheta^{(k)}_{sos}= \vartheta^*$
for all $k$ that is large enough.
For a general polynomial optimization problem, if $\idl[\Phi]+\qmod[\Psi]$ is archimedean,
then $\vartheta^{(k)}_{mom}\to \vartheta^*$ as $k\to \infty$ \cite{Las01}.
In Step~2, the rank condition~\reff{eq:flatrank} is called \textit{flat truncation} \cite{nie2013certifying}.
It is a sufficient (and almost necessary) condition to check
the finite convergence of moment relaxations.
When \reff{eq:flatrank} holds, the method in \cite{HenLas05}
can be used to extract $r$ minimizers for \reff{eq:genopt}.
This method and Algorithm~\ref{ag:momsos} are implemented in the software {\tt GloptiPoly 3}~\cite{GloPol3}.
In the following subsections,
we study the convergence result of Algorithm~\ref{ag:momsos}
when it is applied for solving (\ref{eq:lowerlvopt}),
(\ref{eq:upperlvopt}), \reff{eq:upperlvoptall} and \reff{eq:maxxi}.

\subsection{The optimization for all players}

We discuss the convergence of Algorithm~\ref{ag:momsos}
for solving (\ref{eq:upperlvopt}), \reff{eq:upperlvoptall} and \reff{eq:maxxi}.

First, we consider \reff{eq:upperlvopt}.
Let
\be\label{eq:z=x1}z\, \coloneqq \,x,\quad \theta(x)\, \coloneqq \,[x]_1^T\Theta[x]_1,\ee
and we denote the polynomial tuples
\be \label{poly:Phi:ith}
\begin{array}{l}
\Phi_i \, \coloneqq \, \Big \{\nabla_{x_i} f_i(\allx)
  -\sum_{j=1}^{m_i}\lambda_{i,j}(\allx)\nabla_{x_i} g_{i,j}(\xpi) \Big \} \\
\qquad\qquad\qquad\  \cup \Big \{ g_{i,j}(\xpi): j \in \mcal{E}_i \Big  \}
 \cup \Big \{\lambda_{i,j}(\allx)\cdot g_{i,j}(\xpi): j \in \mcal{I}_i \Big  \} ,
\end{array}
\ee
\be \label{poly:Psi:ith}
\begin{array}{l}
\Psi_i \, \coloneqq \,  \Big \{ g_{i,j}(\xpi): j \in \mcal{I}_i \Big  \} \cup
\Big \{\lambda_{i,j}(\allx): j \in \mcal{I}_i \Big  \} \\
\qquad\qquad\qquad\qquad\qquad  \cup \Big \{f_i(v,\xmi)-f_i(\xpi,\xmi): \, v \in \mcal{K}_i  \Big \}  .
\end{array}
\ee
In the above, for a vector $p = (p_1,\ldots, p_s)$ of polynomials,
the set $\{ p \}$ stands for $\{p_1, \ldots, p_s\}$, for notational convenience.
Denote the unions
\be \label{union:Phi:Psi}
\Phi \,  \coloneqq \bigcup_{i=1}^N \Phi_i,  \quad
\Psi \, \coloneqq \,\bigcup_{i=1}^N \Psi_i.
\ee
They are both finite sets of polynomials.
Then, the optimization (\ref{eq:upperlvopt}) can be written as (\ref{eq:genopt}),
and we may apply Algorithm~\ref{ag:momsos} for solving it.
Recall that $e_i$ is the vector in $\re^n$ such that
its $i$th entry is $1$ and all other entries are zero.
For a tms $y\in\re^{\N^n_{2k}}$, the $y_{e_i}$
means the entry of $y$ labelled by $e_i$.
For example, when $n=4$, $y_{e_2} = y_{0100}$.
Let $y^{(k)}$ be a minimizer of the $k$th order moment relaxation (\ref{eq:d-mom})
for (\ref{eq:upperlvopt}), and denote
\be\label{eq:uk}u^{(k)}\, \coloneqq \,(y^{(k)}_{e_1},y^{(k)}_{e_2}\ddd y^{(k)}_{e_n}).\ee
Then, $u^{(k)}$ is a minimizer of (\ref{eq:upperlvopt}) if $u^{(k)}$ is feasible for (\ref{eq:upperlvopt}) and $\lip \theta, y^{(k)}\rip=\theta(u^{(k)})$.
Moreover, we have the following convergence result for solving (\ref{eq:upperlvopt}):
\begin{theorem} \label{thmcvg:upconv}
For the polynomial optimization problem (\ref{eq:upperlvopt}),
assume $\Theta$ is a generic positive definite matrix.
Let $z \coloneqq x$, and let $\theta,\Psi,\Phi$ be given as in
(\ref{eq:z=x1})-(\ref{union:Phi:Psi}).
Suppose $\idl[\Phi]+\qmod[\Psi]$ is archimedean.
\bit

\item [(i)]
If the optimization \reff{eq:upperlvopt} is infeasible,
then the moment relaxation \reff{eq:d-mom}
must be infeasible when the order $k$ is big enough.

\item [(ii)]
Suppose the optimization \reff{eq:upperlvopt} is feasible.
Let $u^{(k)}$ be given as in (\ref{eq:uk}).
Then $u^{(k)}$ converges to the unique minimizer of \reff{eq:upperlvopt}.
In particular, if the real zero set of $\Phi$ is finite,
then $u^{(k)}$ is the unique minimizer of \reff{eq:upperlvopt}
and (\ref{eq:flatrank}) holds at $y^{(k)}$
with the rank equals $1$ when $k$ is sufficiently large.

\eit
\end{theorem}
\begin{proof}
(i) If \reff{eq:upperlvopt} is infeasible,  the constant polynomial $-1$
can be viewed as a positive polynomial on the feasible set of \reff{eq:upperlvopt}.
Since $\idl[\Phi]+\qmod[\Psi]$ is archimedean,
we have $-1 \in  \idl[\Phi]_{2k}+\qmod[\Psi]_{2k}$,
for $k$ big enough, by Putinar's Positivstellensatz \cite{putinar1993positive}.
For such a big $k$, the SOS relaxation \reff{eq:d-sos} is unbounded from above,
hence the moment relaxation \reff{eq:d-mom} must be infeasible.

(ii) When the optimization \reff{eq:upperlvopt} is feasible,
it must have minimizers.
Let $K$ be the feasible set of \reff{eq:upperlvopt},
and
\[\mcal{R}_2(K)\, \coloneqq \, cone(\{[u]_2: u\in K\}).\]
In the above, the $cone$ means the conic hull.
Consider the moment optimization problem
\be\label{eq:momopt}
\left\{ \begin{array}{cl}
\min\limits_{w} & \lip\theta, w\rip\\
\st & y_{0}=1,\ w\in\mcal{R}_2(K).
\end{array} \right.
\ee
If the matrix $\Theta$ is a generic positive definite matrix,
then the function $\theta$ is generic in $\Sigma_{n,2}$.
By \cite[Proposition~5.2]{nie2014moment},
the moment optimization problem (\ref{eq:momopt}) has a unique minimizer.
When (\ref{eq:momopt}) has minimizers,
its minimum value equals $\vartheta^*$.
Suppose (\ref{eq:upperlvopt}) has two distinct minimizers,
say, $\allx^{(1)}$ and $\allx^{(2)}$.
Then, $[\allx^{(1)}]_2$ and $[\allx^{(2)}]_2$
are two distinct minimizers of (\ref{eq:momopt}),
a contradiction to the uniqueness of the minimizer for (\ref{eq:momopt}).
Therefore, (\ref{eq:upperlvopt}) must have a unique minimizer $x^*$
when $\Theta$ is generic.

The convergence of $u^{(k)}$ to $\allx^*$ is shown in
\cite{Swg05} or \cite[Theorem~3.3]{nie2013certifying}.
For the special case that $\Phi(\allx)=0$ has finitely many real solutions,
the point $u^{(k)}$ must equal $\allx^*$, when $k$ is large enough.
This is shown in \cite{lasserre2008semidef} (also see \cite{nie2013polynomial}).
\end{proof}

The archimedeanness of $\idl[\Phi]+\qmod[\Psi]$
is essentially requiring that the feasible set of (\ref{eq:upperlvopt})
is compact.
If the real zero set of $\Phi$ is compact,
then $\idl[\Phi]+\qmod[\Psi]$ must be archimedean.
In particular, if the NEPP has finitely many real KKT points,
then $\idl[\Phi]+\qmod[\Psi]$ is archimedean.
Interestingly, when the objective and constraining polynomials are generic,
there are finitely many KKT points.
See Theorem~\ref{tm:genfinite} in the appendix.
In fact, as shown in the proof of Theorem~\ref{tm:genfinite},
the zero set of $\Phi$ is finite
for generic NEPPs, and hence Algorithm~\ref{ag:momsos} has finite convergence.
Moreover, by Theorem~\ref{thmcvg:upconv},
when $\Theta$ is generic and the minimizer $y^{(k)}$
for (\ref{eq:d-mom}) is obtained, one may let $u^{(k)}$
be given as in (\ref{eq:uk}) and directly check if $u^{(k)}$
is the unique minimizer or not,
instead of checking the flat truncation (\ref{eq:flatrank}).

The other minimization problem \reff{eq:upperlvoptall}
can be solved in the same way by Algorithm~\ref{ag:momsos}.
The convergence property is the same.
For the cleanness of the paper, we omit the details.

For the maximization \reff{eq:maxxi},
we let $z \coloneqq  x$ and
\be\label{eq:zx2}
\theta(x)\, \coloneqq \, -[x]_1^T \Theta [x]_1.\ee
Recall that the polynomial tuples $\Phi_i$ and $\Psi_i$ are given by (\ref{poly:Phi:ith}-\ref{poly:Psi:ith}).
Denote the set of polynomials
\be \label{newdef:Psi}
\Phi \,  \coloneqq \, \bigcup_{i=1}^N \Phi_i , \quad
\Psi \,  \coloneqq \, \bigcup_{i=1}^N \Psi_i \cup
\Big\{
[\allx^*]_1^T \Theta [\allx^*]_1 + \dt - [\allx]_1^T \Theta [\allx]_1
\Big\} .
\ee
Then \reff{eq:maxxi} can be equivalently written as (\ref{eq:genopt}).
Similarly, Algorithm~\ref{ag:momsos} can be used to solve \reff{eq:maxxi}.
The optimization \reff{eq:maxxi} is always feasible because $\allx^*$ is a feasible point.
Therefore, the moment relaxation \reff{eq:d-mom} is also feasible,
and there is no need to check its feasibility in Step~1 of Algorithm~\ref{ag:momsos}.
Since the minimum value $\vartheta^{(k)}_{mom}$ is a lower bound of $\vartheta^*$,
if $\vartheta^{(k)}_{mom}\ge -[\allx^*]_1^T \Theta [\allx^*]_1$, then
\[
\vartheta^{(k)}_{mom}=\vartheta^*=-[\allx^*]_1^T \Theta [\allx^*]_1,
\]
and $\allx^*$ is a maximizer of \reff{eq:maxxi}.
When $\vartheta_k< -[\allx^*]_1^T \Theta [\allx^*]_1$, the flat truncation condition~\reff{eq:flatrank} can be applied for
checking the finite convergence of the Moment-SOS hierarchy.
Under some classical optimality conditions,
we have $\vartheta^{(k)}_{mom}=\vartheta^*$ when $k$
is large enough \cite{nie2014optimality}.
Moreover, if the real zero set of $\Phi$ is finite,
then the Moment-SOS hierarchy has finite convergence and (\ref{eq:flatrank}) holds \cite{nie2013polynomial}.
We would like to remark that when the NEP is given by generic polynomials,
the complex zero set of $\Phi$ is finite (see Theorem~\ref{tm:genfinite}), thus Algorithm~\ref{ag:momsos} has finite convergence.

\subsection{Checking Nash equilibria}
\label{sc:loweropt}

Suppose $u$ is a minimizer of \reff{eq:upperlvopt}.
To check if $u=(u_i,u_{-i})$ is an NE or not, we need to solve the individual
optimization \reff{eq:lowerlvopt} for all $i\in [N]$.

For the given $u\in\re^n$ and $i\in[N]$,
\reff{eq:lowerlvopt} is a polynomial optimization problem in the variable $x_i$.
If (\ref{eq:lowerlvopt}) is unbounded from below, then $u$ cannot be an NE,
and the point $v$ for precluding $u$ can be obtained by
adding a suitable extra ball constraint.
In the following, we suppose that the minimum of (\ref{eq:lowerlvopt}) is attainable.
Since we assume that the polynomial tuple $g_i(\xpi)$ is nonsingular,
polynomial expressions for Lagrange multiplier expressions exist
and can be applied to solve (\ref{eq:lowerlvopt}).
Let $\lmd_i(x)$ be the Lagrange multiplier expressions in (\ref{polyexpr:lmdij}).
Note that the nonsingularity of $g_i$ implies the LICQC holds at every $x_i\in X_i$.
Every minimizer of (\ref{eq:lowerlvopt})
must be a KKT point of (\ref{eq:lowerlvopt}).
Therefore, (\ref{eq:lowerlvopt}) is equivalent to
the following polynomial optimization problem:
\be  \label{eq:rewtlower}
\left\{
\baray{rcl}
\omega_i \, \coloneqq \,  &\min\limits_{x_i \in \re^{n_i} }  & f_i(\xpi, u_{-i})-f_i(u_i,u_{-i}) \\
&\st & \nabla_{x_i}f_i(\xpi, u_{-i}) -
 \sum\limits_{j=1}^{m_i}\lambda_{i,j}(\xpi, u_{-i} ) \nabla_{x_i} g_{i,j}(\xpi)=0, \\
    && g_{i,j}(\xpi, u_{-i})=0\ (j\in\mcal{E}_i),\\
    && g_{i,j}(\xpi, u_{-i})\lmd_{i,j}(\xpi, u_{-i})=0 \ (j\in\mcal{I}_i),\\
    && g_{i,j}(\xpi, u_{-i})\ge0\ (j\in\mcal{I}_i),\\
    && \lmd_{i,j}(\xpi, u_{-i})\ge0\ (j\in\mcal{I}_i).
\earay
\right.
\ee
We introduce the convergence result of Algorithm~\ref{ag:momsos}
for solving (\ref{eq:rewtlower}).
Let
\be\label{eq:zx3}
z\, \coloneqq \, x_i,\quad \theta(x_i)\, \coloneqq \, f_i(x_i, u_{-i})-f_i(u_i, u_{-i}),
\ee
\be \label{polylow:Hi(u)}
\begin{aligned}
\Phi \, \coloneqq \, &\big \{g_{i,j}(\xpi): j \in \mcal{E}_i \big \} \cup
\big \{ \lambda_{i,j}(\xpi,u_{-i})\cdot g_{i,j}(\xpi): j \in \mcal{I}_i \big \}  \\
& \qquad \qquad\qquad\cup \big \{ \nabla_{x_i}f_i(\xpi, u_{-i}) -
 \sum\nolimits_{j=1}^{m_i}\lambda_{i,j}(\xpi, u_{-i} ) \nabla_{x_i} g_{i,j}(\xpi) \big \} ,
\end{aligned}
\ee
\be  \label{polylow:Gi(u)}
\Psi \, \coloneqq \, \big \{g_{i,j}(\xpi): j \in \mcal{I}_i \big \} \cup
\big \{ \lambda_{i,j}(\xpi,u_{-i}) : j \in \mcal{I}_i \big \}   .
\ee
Like earlier cases, the set $\{ p \}$ stands for $\{p_1, \ldots, p_s\}$,
when $p = (p_1,\ldots, p_s)$ is a vector of polynomials.
Then, the (\ref{eq:rewtlower}) can be rewritten as (\ref{eq:genopt}),
and the Moment-SOS relaxations of (\ref{eq:rewtlower}) are given by (\ref{eq:d-sos}) and (\ref{eq:d-mom}).
We would like to remark that the optimization \reff{eq:rewtlower}
is always feasible since $u_i$ is in its feasible set.
Thus the moment relaxation (\ref{eq:d-mom}) for \reff{eq:rewtlower}
is also feasible, and there is no need to check the feasibility for (\ref{eq:d-mom})
in the first step of Algorithm~\ref{ag:momsos}.
Moreover, the minimum $\vartheta^{(k)}_{mom}$ of (\ref{eq:d-mom})
is a lower bound for $\omega_i$, and $\omega_i\le 0$.
If $\vartheta^{(k)}_{mom}\ge0$ for some $k\ge d_0$,
then $\omega_i$ must be $0$, and we can stop Algorithm~\ref{ag:momsos}
immediately because this implies that $u_i$ is the minimizer for $\mathbf{F}_i(u_{-i})$.
If $\vartheta^{(k)}_{mom}<0$, we need to apply the flat truncation (\ref{eq:flatrank})
to certify if the finite convergence for the Moment-SOS hierarchy is achieved or not.
The following theorem concerns the finite convergence of Algorithm~\ref{ag:momsos}
for solving (\ref{eq:rewtlower}).
Its proof follows from \cite[Theorem~4.4]{Nie2018saddle}.

\begin{theorem}  \label{thmcvg:lowopt}
Assume the $i$th player's constraining polynomial tuple $g_i$
is nonsingluar and its optimization \reff{eq:lowerlvopt}
has a minimizer for the given $u_{-i}$.
Let $z \coloneqq x_i$, and let $\theta,\Psi,\Phi$
be given as in (\ref{eq:zx3})-(\ref{polylow:Gi(u)}).
Assume either one of the following conditions hold:
\bit

\item [(i)]
The set $\idl[\Phi]+\qmod[\Psi]$ is archimedean,

\item [(ii)]
The real zero set of polynomials in $H_i(u)$ is finite.

\eit
If each minimizer of \reff{eq:rewtlower} is an isolated critical point,
then all minimizers of \reff{eq:d-mom}
must satisfy the flat truncation \reff{eq:flatrank},
for all $k$ big enough.
Therefore, Algorithm~\ref{ag:momsos} must terminate within
finitely many loops.
\end{theorem}

We remark that if $\idl[g_{i,j}: j \in \mcal{E}_i] +
\qmod[g_{i,j}: j \in \mcal{I}_i]$ is archimedean,
then $\idl[\Phi]+\qmod[\Psi]$ is also archimedean.
Therefore, if the archimedeanness holds for
the $i$th player's optimization \reff{eq:subprob},
then the condition (i) in Theorem~\ref{thmcvg:lowopt} is satisfied.

\section{Numerical Experiments}
\label{sc:num}

This section reports numerical experiments for solving NEPs
by Algorithms~\ref{ag:KKTSDP} and \ref{ag:allKKTSDP}.
For all polynomial optimization problems appearing in the algorithms,
we apply the software {\tt GloptiPoly~3} \cite{GloPol3}
to formulate Moment-SOS semidefinite relaxations,
and use {\tt SeDuMi} \cite{sturm1999using}
for solving these semidefinite programs.
The computation is implemented in an Alienware Aurora R8 desktop,
with an Intel\textsuperscript \textregistered \
Core(TM) i7-9700 CPU at 3.00GHz$\times $8 and 16GB of RAM,
in a Windows 10 operating system.

For ball and simplex constraints,
the expressions are given by (\ref{eq:simplme}) and (\ref{eq:balllme}) respectively.
Polynomial expressions of Lagrange multipliers for
other types of constraints are given in the descriptions of each example.
In Step~2 of Algorithm~\ref{ag:KKTSDP} and Step~3 of Algorithm~\ref{ag:allKKTSDP},
if the optimal value $\omega_i \ge 0$ for all players, then the point $u$ is an NE.
In numerical computation, we cannot have $\omega_i \ge 0$ exactly,
due to round-off errors.
Therefore, we use the parameter
\[
\omega^* \,  \coloneqq  \, \min_{i=1,\ldots,N} \omega_i
\]
to measure the accuracy of the computed NE.
Typically, if $\omega^*$ is small, say, $\omega^* \ge -10^{-6}$,
then we regard the computed solution as an NE.

\begin{example}\rm
For the convex NEP in Example~\ref{ep:isolatedep},
Algorithm~\ref{ag:KKTSDP} found the NE
\[x_1^*=(-1.0000, 0.0000),\quad x_2^*=(0.4472, 0.8944)\]
in the first loop, as shown in Example~\ref{ep:cvxcomp}.
The accuracy parameter is $\omega^*=-7.9793\cdot 10^{-9}$.
Then,
we ran Algorithm~\ref{ag:allKKTSDP} and found two more NEs,
which are
\[\begin{array}{ll}x^*_1=(-0.0000,0.0000),\  x^*_2=(0.0000,0.0000), &\
 \omega^* = - 1.4147\cdot 10^{-10};\\
x^*_1=(1.0000,-0.0000),\  x^*_2=(-0.4472,-0.8944), &\
 \omega^* = - 1.7829\cdot 10^{-8}.\end{array}\]
Moreover, Algorithm~\ref{ag:allKKTSDP}
certified that these three NEs are all solutions to this NEP.
It took around 1.40 seconds to find these two additional NEs
and certify the completeness of the solution set.
\end{example}

In the following example, we show that our algorithm can find NEs
for NEPs which have infinitely many KKT points.
\begin{example}\rm
\label{ep:infiKKT}
(i) Consider the convex NEP
\begin{equation*}
\mbox{1st player:} \left\{
\begin{array}{cl}
\min\limits_{x_1 \in \re^2} &  (x_{1,1}+x_{1,2}-x_{2,1}-x_{2,2})^2,  \\
\st & 1-(x_{1,1})^2-(x_{1,2})^2 \ge 0,
\end{array}\right.
\end{equation*}
\begin{equation*}
\mbox{2nd player:}   \left\{
\begin{array}{cl}
\min\limits_{x_2 \in \re^2}& (x_{2,1}-x_{1,1})^2+(1-x_{2,2})^2 \\
s.t. & 1-x_{2,1}-x_{2,2} \ge 0, \, x_{2,1} \ge 0, \, x_{2,2} \ge 0,
\end{array}\right.
\end{equation*}
then one may check that for each $\alpha\in[0,1/2]$,
$x_1=(2\alpha,1-2\alpha)$, $x_2=(\alpha,1-\alpha)$ is an NE.
Applying Algorithm~\ref{ag:KKTSDP},
we got the NE:
\[
x_1^* =(0.9247,0.0753),\  x_2^* = (0.4624, 0.5376), \ \omega^* = -2.1940\cdot 10^{-8}.
\]
The computation took about 0.19 second.

(ii) Consider the NEP
\begin{equation*}
\mbox{1st player:} \left\{
  \begin{array}{cl}
     \min\limits_{x_1 \in \re^2} &  -x_{2,1}(x_{1,1})^2-x_{2,2}x_{1,1}+(x_{2,2}-\frac{1}{2})x_{1,2}  \\
\st & 1-(x_{1,1})^2-(x_{1,2})^2 \ge 0,
  \end{array}\right.
\end{equation*}
\begin{equation*}
\mbox{2nd player:}   \left\{
\begin{array}{cl}
\min\limits_{x_2 \in \re^2}& x_{1,2}x_{2,1}+(x_{2,2}-\frac{1}{2})^2 \\
    s.t. & 1-x_{2,1}-x_{2,2} \ge 0, \, x_{2,1} \ge 0, \, x_{2,2} \ge 0.
  \end{array}\right.
\end{equation*}
One may check that for each $\alpha\in[-\frac{1}{4},0)$,
the $x_1=(\alpha,0)$, $x_2=(-\frac{1}{4\alpha},\frac{1}{2})$
is a KKT point that is not an NE.
Applying Algorithm~\ref{ag:KKTSDP}, we got the NE
\[
x_1^* =(1.0000,-0.0000),\ x_2^* = (0.4259, 0.5000), \ \omega^* = -6.2187\cdot 10^{-9}.
\]
The computation took about 0.33 second.
\end{example}

\begin{example}
\label{ep:oracle}
\rm
In this example, we consider NEPs with box constraints such that every $x_i\in\re^1$.
For each $i\in[N]$, the $i$th player's feasible set is given by
\[1+x_i\ge 0,\quad  1-x_i\ge0.\]
Then, the associated Lagrange multipliers can be expressed as
\[\lmd_{i,1} = \frac{1}{2}\frac{\partial f_i}{\partial x_{i}} \cdot (1-x_{i}),\quad
\lmd_{i,2} = \lmd_{i,1} - \frac{\partial f_i}{\partial x_{i}}.\]

(i) Consider the two-player zero-sum game with box constraints in \cite[Example~3.1]{parrilo2006polynomial}
(see also \cite[Example~1]{kroupa2023multiple}), where the objective functions are
\[f_1(x_1,x_2) = (x_1)^2-2x_1(x_2)^2  +x_2,\quad f_2(x_1,x_2) = -f_2(x_1,x_2). \]
Applying Algorithm~\ref{ag:KKTSDP}, we got the NE:
\[ x_1^* = 0.3969,\ x_2^* = 0.6300,\ \omega^* = -2.9179\cdot 10^{-11}  \]
in the initial loop.
It took around 0.54 second.

(ii) Consider the two-player game with box constraints in
\cite[Example~2.3]{stein2008separable} (see also \cite[Example~2]{kroupa2023multiple}),
where the objective functions are
\[\begin{aligned}f_1(x_1,x_2) = \, &2(x_1)^3+3(x_1x_2)^2-2x_1x_2+x_1-3(x_2)^3, \\
f_2(x_1,x_2) = \, & 4(x_2)^3 -2(x_1x_2)^2+(x_1)^2-(x_1)^2x_2-4x_2. \end{aligned}\]
Applying Algorithm~\ref{ag:KKTSDP}, we detected nonexistence of NEs in the third loop\footnote{We remark that for this NEP, as well as the NEP in Example~\ref{ep:oracle}(iii), though a (pure strategy) NE does not exist, there exist mixed strategy solutions. See \cite{kroupa2023multiple,stein2008separable} for more details.}.
It took around 0.85 second.

(iii) Consider the generalization of separable network games in
\cite[Example~5]{kroupa2023multiple}.
The objective functions are
\[\begin{aligned}
f_1(x_1,x_2,x_3) = \, & 2(x_1)^2+2x_1(x_2)^2 - 5x_1x_2 +4x_1x_3+x_2+2x_3, \\
f_2(x_1,x_2,x_3) = \, & 2(x_2)^2-2x_1(x_2)^2+5x_1x_2-5x_2x_3+2x_2(x_3)^2-x_2+2(x_1)^2,\\
f_3(x_1,x_2,x_3) = \, & -2x_2(x_3)^2-4x_1x_3+5x_2x_3-2x_3-4(x_1)^2-2(x_2)^2. \end{aligned}\]
Applying Algorithm~\ref{ag:KKTSDP}, we detected nonexistence of NEs in the second loop.
It took around 0.90 second.
\end{example}

For all NEPs in the following examples except Example~\ref{ep:bs},
our method found all NEs with certified completeness of solution sets.
In the following, we only report the numerical result of finding all solutions,
unless specifically mentioned, for the neatness of this paper.

\begin{example}
\label{ep:nonconvex2x3}
\rm
Consider the $2$-player NEP
\begin{equation*}
\mbox{1st player:} \left\{
  \begin{array}{cl}
    \min\limits_{x_1 \in \re^3} & \sum_{j=1}^3 x_{1,j}(x_{1,j}- j\cdot x_{2,j})\\
    s.t. & 1-x_{1,1}x_{1,2} \ge 0, \,1-x_{1,2}x_{1,3}\ge 0,\, x_{1,1} \ge 0,
  \end{array}\right.
\end{equation*}
\begin{equation*}
\mbox{2nd player:}  \left\{
  \begin{array}{cl}
    \min\limits_{x_2 \in \re^3} & 
    \prod_{j=1}^3 x_{2,j} +
    \sum\limits_{\mathclap{\substack{1\le i<j\le 3\\1\le k\le 3}}} x_{1,i}x_{1,j}x_{2,k}+
    \sum\limits_{\mathclap{\substack{1\le i\le 3\\1\le j<k\le 3}}} x_{1,i}x_{2,j}x_{2,k}
    \\
    s.t. & 1-(x_{2,1})^2 - (x_{2,2})^2 = 0.
  \end{array}\right.
\end{equation*}
The first player's optimization is non-convex, with an unbounded feasible set.
The Lagrange multipliers for the first player's optimization are
\be \nn
\begin{array}{lll}
  \lambda_{1,1}=(1-x_{1,1}x_{1,2})\frac{\partial f_1}{\partial x_{1,1}}, &
  \lambda_{1,2}=-x_{1,1}\frac{\partial f_1}{\partial x_{1,2}}, &
  \lambda_{1,3}=x_{1,1}\frac{\partial f_1}{\partial x_{1,1}}
      -x_{1,2}\frac{\partial f_1}{\partial x_{1,2}}.
\end{array}
\ee
Applying Algorithm \ref{ag:allKKTSDP}, we got four NEs:
\[
\begin{array}{ll}
x_1^*=(0.3198,  0.6396, -0.6396 ), &\ x_2^*=( 0.6396,    0.6396,    -0.4264); \\
x_1^*=(0.0000, 0.3895, 0.5842 ), &\ x_2^*=(-0.8346, 0.3895, 0.3895); \\
x_1^*=(0.2934, -0.5578, 0.8803 ), &\ x_2^*=( 0.5869, -0.5578,  0.5869);\\
x_1^*=(0.0000, -0.5774, -0.8660 ), &\ x_2^*=( -0.5774, -0.5774, -0.5774). \\
\end{array}
\]
Their accuracy parameters are respectively
\[
-7.1879\cdot 10^{-8}, \, -3.5040\cdot 10^{-7}, \, -4.3732\cdot 10^{-7},
-6.4360\cdot 10^{-7}.
\]
It took about $30$ seconds.

However, if the second player's objective becomes
\[
-\prod_{j=1}^3 x_{2,j} +
\sum\limits_{\mathclap{\substack{1\le i\le 3\\1\le j<k\le 3}}} x_{1,i}x_{2,j}x_{2,k}
-\sum\limits_{\mathclap{\substack{1\le i<j\le 3\\1\le k\le 3}}} x_{1,i}x_{1,j}x_{2,k},
\]
then there is no NE, which was detected by Algorithm~\ref{ag:KKTSDP}.
It took around $16$ seconds.
\end{example}

\begin{example} \rm
\label{ep:plain3x2ball_arc_box}
Consider the $3$-player NEP
\begin{equation*}
\mbox{1st player:}  \left\{
 \begin{array}{cl}
 \min\limits_{x_1 \in \re^2} &  (2x_{1,1}-x_{1,2}+3)x_{1,1}x_{2,1}\\
              &\qquad \qquad \qquad+[(2x_{1,2})^2+(x_{3,2})^2]x_{1,2}\\

\st  & 1 - x_1^Tx_1 \ge 0,
    \end{array}\right.
\end{equation*}
\begin{equation*}
\mbox{2nd player:}  \left\{
 \begin{array}{cl}
 \min\limits_{x_2 \in \re^2 } &  [(x_{2,1})^2-x_{1,2}]x_{2,1} \\
 &\qquad \quad+[(x_{2,2})^2+2x_{3,2}+x_{1,2}x_{3,1}]x_{2,2}\\
 \st &x_2^Tx_2-1 =0, \, x_{2,1} \ge 0, \, x_{2,2} \ge 0,
 \end{array}\right.
\end{equation*}
\begin{equation*}
\mbox{3rd player:}  \left\{
 \begin{array}{cl}
 \min\limits_{x_3 \in \re^2 }& (x_{1,1}x_{1,2}-1)x_{3,1}- [3(x_{3,2})^2+1]x_{3,2}\\
 &\qquad \qquad \qquad\quad+2[x_{3,1}+x_{3,2}]x_{3,1}x_{3,2} \\
\st & 1-(x_{3,1})^2 \ge 0, \, 1 - (x_{3,2})^2 \ge 0.
 \end{array}\right.
\end{equation*}
The Lagrange multipliers can be represented as
\[
\begin{array}{lll}
    \lambda_{2,1}=\frac{1}{2}(x_2^T\nabla_{x_2} f_2),
  & \lambda_{2,2}=\frac{\partial f_2}{x_{2,1}}-2x_{2,1}\lambda_{2,1},
  &\lambda_{2,3}=\frac{\partial f_2}{x_{2,2}}-2x_{2,2}\lambda_{2,1},\\
    \lambda_{3,1}=-\frac{x_{3,1}}{2}\frac{\partial f_3}{\partial x_{3,1}},
  &\lambda_{3,2}=-\frac{x_{3,2}}{2}\frac{\partial f_3}{\partial x_{3,2}}.
\end{array}
\]
Applying Algorithm~\ref{ag:allKKTSDP},  we got the unique NE
\[
x_1^* = (-0.3558, -0.9346), \
x_2^* = (1.0000,   0.0000), \
x_3^* = (-0.3331,  1.0000).
\]
The accuracy parameter is $-9.2310\cdot 10^{-9}.$
It took around $9$ seconds.

Nonetheless, if the third player's objective becomes $-f_1(\allx)-f_2(\allx)$,
then the NEP becomes a zero-sum game and there is no NE,
which was detected by Algorithm~\ref{ag:KKTSDP}.
It took around $3$ seconds.
\end{example}

\begin{example} \rm
\label{ep:ring}
Consider the $2$-player NEP
\begin{equation*}
\mbox{1st player:} \left\{
\begin{array}{cll}
  \min\limits_{x_1 \in \re^2} &  2x_{1,1}x_{1,2}+
  3x_{1,1}(x_{2,1})^2+3(x_{1,2})^2x_{2,2}\\
 \st &(x_{1,1})^2+(x_{1,2})^2-1\ge 0,  \\
     &2-(x_{1,1})^2-(x_{1,2})^2\ge 0,
  \end{array}\right.
\end{equation*}
\begin{equation*}
\mbox{2nd player:} \left\{
\begin{array}{cll}
  \min\limits_{x_2 \in \re^2 }& (x_{2,1})^3+(x_{2,2})^3+x_{1,1}(x_{2,1})^2 \\
  & \qquad \qquad +x_{1,2}(x_{2,2})^2 +x_{1,1}x_{1,2}(x_{2,1}+x_{2,2}) \\
  \st &(x_{2,1})^2+(x_{2,2})^2-1 \ge 0, \\
      &2-(x_{2,1})^2+(x_{2,2})^2 \ge 0.
  \end{array}\right.
\end{equation*}
The Lagrange multipliers can be represented as ($i=1,2$):
\[
\lambda_{i,1}=\frac{1}{2}\nabla_{x_i} f_i^Tx_i(2-x_i^Tx_i),\quad
 \lambda_{i,2}=\frac{1}{4}\nabla_{x_i} f_i^Tx_i(1-x_i^Tx_i).
 \]
By Algorithm \ref{ag:allKKTSDP}, we got the unique NE
\[
x_1^* =(-1.3339,0.4698), \quad x_2^* = (-1.4118,0.0820),
\]
with the accuracy parameter $-3.5186\cdot 10^{-8}.$
It took around $5$ seconds.
\end{example}

\begin{example}
\label{ep:bs}
\rm
Consider the NEP
\begin{equation*}
\baray{l}
\mbox{1st player:}  \left\{
\begin{array}{cl}
   \min\limits_{x_1 \in \re^{n_1}} &
   \sum_{1 \le i \le j \le n_1} x_{1,i}x_{1,j}(x_{2,i}+x_{2,j})  \\
\st   & 1-(x_{1,1}^2+\cdots+x_{1,n_1}^2) = 0,
\end{array}\right.
\\
\mbox{2nd player:}   \left\{
\begin{array}{cl}
    \min\limits_{x_2 \in \re^{n_2}} &
    \sum_{1 \le i\le j \le n_2} x_{2,i}x_{2,j}(x_{1,i}+x_{1,j})   \\
\st & 1-(x_{2,1}^2+\cdots+x_{2,n_2}^2) = 0,
\end{array}\right.
\earay
\end{equation*}
where $n_1=n_2$.
We ran Algorithm~\ref{ag:allKKTSDP} for cases $n_1=n_2=3,4,5,6$.
The computational results are shown in Table~\ref{tab:exm:bs}.
\begin{table}[htbp]
\centering
\caption{Computational results for Example~\ref{ep:bs}}
\label{tab:exm:bs}
\begin{tabular}{|l|c|c|r|}  \hline
$n_1$  & NE  & $\omega^*$   & time  \\ \hline
$3$  & $\begin{array}{l} x_1^* = (-0.5774,\, -0.5774, \, -0.5774)\\ x_2^* = (-0.5774,\, -0.5774, \, -0.5774) \end{array} $ & $-1.0689\cdot 10^{-7}$ & 1.31  \\ \hline
$4$  & $\begin{array}{l} x_1^* = (0.8381, \,0.5024,\\ \qquad\qquad \qquad  -0.0328 , \,-0.2098) \\ x_2^* = (-0.1791, \, -0.0683,\\ \qquad\qquad\qquad\quad   0.4066, \,0.8933) \end{array} $ & $-1.4459\cdot 10^{-9}$ & 62.85  \\ \hline
$5$  & $\begin{array}{l} x_1^* = (0.8466,\, 0.4407,\, 0.1744,\\ \qquad\qquad \qquad  -0.0101,\, -0.2418) \\ x_2^* = (-0.1944,\,-0.0512, \, 0.1238,\\ \qquad\qquad\qquad\quad   0.3370, \, 0.9114) \end{array} $ & $-2.7551\cdot 10^{-9}$ & 682.67  \\ \hline
$6$  & $\begin{array}{l} x_1^* = (0.8026,\, 0.4724,\, 0.1799,\\ \qquad\qquad 0.1799, \, -0.0637,\, -0.2527) \\ x_2^* = (-0.1979,\, -0.0772,\, 0.1091,\\ \qquad\qquad 0.1091, \, 0.4040, \, 0.8762) \end{array} $ & $-7.0354\cdot 10^{-9}$ & 18079.99  \\ \hline
\end{tabular}
\end{table}
In the table, $n_1$ is the dimension for variables $x_1$ and $x_2$,
the column `NE' shows the computed solutions to the NEP,
and $w^*$ is the accuracy parameter.
All time consumptions are displayed in seconds,
Because of the relatively large amount of computational time,
we only compute one NE for each case above.
\end{example}

We would like to remark that our method can also be applied to
solve unconstrained NEPs where all
individual optimization problems have no constraints,
or equivalently, the feasible set $X_i$ for \reff{eq:subprob}
is the entire space $\re^{n_i}$. For unconstrained NEPs,
the KKT system \reff{eq:KKTwithLM} becomes
\[
\nabla_{x_i} f_i(x^*) \, = \, 0, \quad i=1,\ldots, N,
\]
and Algorithms~\ref{ag:KKTSDP} and \ref{ag:allKKTSDP}
can be implemented in the same way.

\begin{example}
\label{ep:unconstrained}
\rm
Consider the unconstrained NEP
\begin{equation*}
\baray{l}
\mbox{1st player:}   \left\{
\begin{array}{cl}
   \min   &  \sum\limits_{i=1}^{n_1} (x_{1,i})^4 +
    \sum\limits_{0\le i\le j \le k  \le n_1}
     \frac{x_{1,i}x_{1,j} ( x_{1,k} + x_{2,i}+ x_{3,j})}{(n_1)^2} \\
 \st & x_1 \in \re^{n_1} ,
\end{array} \right.
\\
\mbox{2nd player:}  \left\{
\begin{array}{cl}
   \min  &  \sum\limits_{i=1}^{n_2} (x_{2,i})^4 +
   \sum\limits_{0\le i\le j \le k  \le n_2}
     \frac{x_{2,i}x_{2,j} ( x_{2,k} + x_{3,i}+ x_{1,j} )}{(n_2)^2} \\
\st   &   x_2 \in \re^{n_2},
\end{array}\right.
\\
\mbox{3rd player:}  \left\{
\begin{array}{cl}
   \min  &  \sum\limits_{i=1}^{n_3} (x_{3,i})^4 +
   \sum\limits_{0\le i\le j \le k  \le n_3}
     \frac{x_{3,i}x_{3,j} ( x_{3,k} + x_{1,i}+ x_{2,j} ) }{(n_3)^2} \\
\st   &   x_3 \in \re^{n_3},
\end{array}\right.
\earay
\end{equation*}
where $x_{1,0}=x_{2,0}=x_{3,0}=1$, and $n_1 = n_2 = n_3$.
We implement Algorithm~\ref{ag:allKKTSDP} for the cases $n_1=n_2=n_3=2,3,4,5,6$.
The computational results are shown in the following table.
For all cases, we computed an NE successfully
and obtained that $x_1^*=x_2^*=x_3^*$ (up to round-off errors).
There is a unique NE for each case.
The computational results are reported in Table~\ref{tab:exm:uncon}.
The time is displayed in seconds.
\begin{table}[htbp]
\caption{The computational results for Example~\ref{ep:unconstrained}.}
\centering
\label{tab:exm:uncon}
\begin{tabular}{|l|c|c|r|}  \hline
$n_1$  & $x_1^* = x_2^* = x_3^*$  & $\omega^*$ & time \\ \hline
$2$ & $(-0.8410,\quad -0.7125)$  & $-8.8291\cdot 10^{-9}$ & 0.34  \\\hline
$3$ & $(-0.6743,-0.6157,-0.5236)$  & $-6.6507\cdot 10^{-9}$ & 1.58  \\\hline
$4$ & $\begin{array}{ll}    (-0.5950,-0.5606\\
   \qquad\qquad-0.5097,-0.4363)\end{array}$
  & $-1.0577\cdot 10^{-9}$ & 16.86  \\\hline
$5$ & $\begin{array}{l}    (-0.5476,-0.5247,-0.4919,\\\qquad\qquad-0.4472,-0.3860)\end{array}$
  & $-4.4438\cdot 10^{-9}$ & 177.63  \\\hline
$6$ & $\begin{array}{l}    (-0.5157,-0.4992,-0.4762,\\-0.4457,-0.4060,-0.3534)\end{array}$
  & $-3.7536\cdot 10^{-9}$ & 1379.27  \\\hline

\end{tabular}
\end{table}
\end{example}

The following are some examples of NEPs from applications.

\begin{example}\rm
Consider the environmental pollution control problem for three countries
for the case \textit{autarky} \cite{breton2006game}.
Let $ x_{i,1} (i= 1,2,3) $ denote the (gross) emissions
from the $i$th country.
The revenue of the $i$th country depends on $ x_{i,1}$,
e.g., a typically one is $x_{i,1}(b_i- \frac{1}{2} x_{i,1})$.
The variable $x_{i,2} $ represents the investment by the $i$th country
to local environmental projects.
The net emission in country $i$ is $x_{i,1}-\gamma_{i}x_{i,2}$,
which is always nonnegative and must be kept below or
equal to a certain prescribed level $E_i>0$
under an environmental constraint.
The damage cost of the $i$th country is assumed to be
$d_i( x_{i,1}-\gamma x_{i,2})+\sum_{j\ne i}c_{i,j}x_{i,2}x_{j,1}.$
For given parameters $b_i,c_{i,j},d_i,\gamma_i,E_i$,
the $i$th ($i=1,2,3$) country's optimization problem is
\[
\left\{ \begin{array}{cl}
 \min\limits_{ x_i \in \re^2} &-x_{i,1}(b_i-\frac{1}{2} x_{i,1})+\frac{(x_{i,2})^2}{2}
            +d_i( x_{i,1}-\gamma_i x_{i,2})+\sum\limits_{j\ne i}c_{i,j}x_{i,2}x_{j,1} \\
    \st &x_{i,2}\ge0,\  x_{i,1}\le b_i,\\
    & 0\le x_{i,1}-\gamma_i x_{i,2}\le E_i.
  \end{array} \right.
\]
We consider the general cases that $b_i\ne E_i$.
The Lagrange multipliers can be expressed as
\be \nn
\begin{array}{llr}
  \lambda_{i,4}=&\frac{1}{(b_i-E_i)E_i}\left(\frac{\partial f_i}{\partial x_{i,2}}x_{i,2}(x_{i,1}
         -\gamma_ix_{i,2})-\frac{\partial f_i}{\partial x_{i,1}}(b_i-x_{i,1})(x_{i,1}-\gamma_ix_{i,2})\right),\\
  \lambda_{i,3}=&\frac{1}{b_i}\left((b_i-x_{i,1})(\frac{\partial f_i}{\partial x_{i,1}}+\lambda_{i,4})-x_{i,2}(\frac{\partial f_i}{\partial x_{i,2}}-\gamma_i\lambda_{i,4}) \right),\\
  \lambda_{i,2}=&\lambda_{i,3}-\lambda_{i,4}-\frac{\partial f_i}{\partial x_{i,1}},\\
  \lambda_{i,1}=&\frac{\partial f_i}{\partial x_{i,2}}+\gamma_i\lambda_{i,3}-\gamma_i\lambda_{i,4}.
\end{array}
\ee
We solve the NEP for the following typical parameters:
\[
\begin{array}{llllll}
b_1=1.5, & b_2=2, & b_3=1.8, & c_{1,2}=0.2, & c_{1,3}=0.3, & c_{2,1}=0.4,\\
c_{2,3}=0.2, & c_{3,1}=0.5, & c_{3,2}=0.1, & d_1=0.8, & d_2=1.2, & d_3=1.0,\\
E_1=3,  & E_2=4, & E_3=2,  & \gamma_1=0.7, & \gamma_2=0.5, & \gamma_3=0.9.
\end{array}
\]
By Algorithm~\ref{ag:allKKTSDP}, we got the unique NE
\[
x_1^*=(0.7000,0.1600), \quad x_2^* = (0.8000,0.1600), \quad x_3^* = (0.8000,0.4700),
\]
with the accuracy parameter $-1.1059\cdot 10^{-9}$.
It took about $10$ seconds.
\end{example}

\begin{example}\rm
Consider the NEP of the electricity market problem \cite{Contreras2004}.
There are three generating companies, and the $i$th
company possesses $s_i$ generating units.
For the $i$th company, the power generation of his
$j$th generating unit is denoted by $x_{i,j}$.
Assume $0\le x_{i,j}\le E_{i,j}$, where the nonzero parameter $E_{i,j}$
represents its maximum capacity, and the cost of this generating unit is $\frac{1}{2}c_{i,j}(x_{i,j})^2+d_{i,j}x_{i,j}$,
where $c_{i,j},d_{i,j}$ are parameters.
The electricity price is given by
$$\phi(\allx) \,  \coloneqq  \, b-a(\sum_{i=1}^3\sum_{j=1}^{s_i}x_{i,j}).$$
The aim of each company is to maximize its profits, that is,
to solve the following optimization problem:
\[
\baray{l}
\mbox{$i$th player:} \left\{
 \begin{array}{cl}
 \min\limits_{x_i \in \re^{s_i}} & \frac{1}{2}\sum_{j=1}^{s_i}(c_{i,j}(x_{i,j})^2+d_{i,j}x_{i,j})-\phi(\allx)(\sum_{j=1}^{s_i}x_{i,j}).\\
\st  & 0\le x_{i,j}\le E_{i,j} \ (j\in[s_i]).
  \end{array}\right.
\earay
\]
The Lagrange multipliers associated to the constraints
$g_{i,2j-1} \coloneqq E_{i,j}-x_{i,j}\ge0, \, g_{i,2j} \coloneqq x_{i,j}\ge0$
can be represented as
$$\lambda_{i,2j-1}=-\frac{\partial f_i}{\partial x_{i,j}}\cdot x_{i,j}/E_{i,j}, \, \lambda_{i,2j}=\frac{\partial f_i}{\partial x_{i,j}}+\lambda_{i,2j-1}.\,(j\in [s_i])$$
For the following parameters
\be \nn
\begin{array}{llllll}
 s_i=i,&&a=1,&& b=10,\\
 c_{1,1}=0.4, & c_{2,1}=0.35, & c_{2,2}=0.35, & c_{3,1}=0.46, & c_{3,2}=0.5, & c_{3,3}=0.5, \\
 d_{1,1}=2,   & d_{2,1}=1.75, & d_{2,2}=1,    & d_{3,1}=2.25, & d_{3,2}=3,   & d_{3,3}=3,  \\
 E_{1,1}=2,   & E_{2,1}=2.5,  & E_{2,2}=0.67, & E_{3,1}=1.2,  & E_{3,2}=1.8, & E_{3,3}=1.6,
\end{array}
\ee
we ran Algorithm~\ref{ag:allKKTSDP} and found the unique NE
\[
x_1^* = 1.7184, \quad x_2^* = (1.8413,0.6700), \quad
x_3^* = (1.2000,0.0823,0.0823).
\]
The accuracy parameter is $-5.1183\cdot 10^{-7}$.
It took about $8$ seconds.
\end{example}

\section{Conclusions and Discussions}
\label{sc:conc}

This paper studies Nash equilibrium problems that are given by polynomial functions.
Algorithms \ref{ag:KKTSDP} and \ref{ag:allKKTSDP} are proposed for
computing one or all NEs.
The Moment-SOS hierarchy of semidefinite relaxations is used
to solve the appearing polynomial optimization problems.
Under generic assumptions, we can compute a Nash equilibrium if it exists,
and detect its nonexistence if there is none.
Moreover, we can get all Nash equilibria if there are finitely many ones of them.

In \cite{Nie2023convex}, a semidefinite relaxation method using rational
and parametric Lagrange multiplier expressions is proposed for solving convex GNEPs.
Under some general conditions, the method in \cite{Nie2023convex}
is guaranteed to find one GNE or detect nonexistence of GNEs.
The NEPs considered in this work are special cases of GNEPs,
since they can be viewed as GNEPs where every player's feasible set
is independent of other players' strategies.
Moreover, for convex NEPs, Algorithm~\ref{ag:KKTSDP} reduces to
\cite[Algorithm~5.3]{Nie2023convex} and terminates at Step~2 in the first loop,
as shown in Corollary~\ref{cor:cvxnep}.
In contrast, this paper mainly focuses on solving nonconvex NEPs,
and the main difficulty of problems
in the scope of this paper is brought by nonconvexity.
Major differences between contributions in this paper
and those in \cite{Nie2023convex} are as follows:

\begin{itemize}

\item In this paper, we primarily focus on nonconvex NEPs of polynomials.
One of our main contributions in this work is that
we proposed an algorithm that finds NEs
for nonconvex NEPs, if they exist.
Note when there is no convexity being assumed, every block $x^*_i$
of the NE $x^*$ is the global minimizer for $\mathbf{F}_i(\xmi^*)$,
which is usually nonconvex.
For nonconvex NEPs, the KKT conditions are typically
not sufficient for global optimality,
thus the updating scheme $\mcal{K}_i\, \coloneqq \,\mcal{K}_i\cup\mcal{U}_i$ in Step~3 of Algorithm~\ref{ag:KKTSDP} is applied to preclude KKT points that are not NEs.
Therefore, we usually need to solve a sequence of
polynomial optimization problems to get NEs.
In comparison, the \cite{Nie2023convex} concerns GNEPs
where every player solves a convex optimization problem.
Therefore, once a KKT point is obtained
with some constraint qualification conditions
being satisfied, this KKT point must be a GNE.
So there is no need to preclude any KKT point,
and we usually only need to solve one polynomial optimization problem for a GNE.
Indeed, convex NEPs are studied in Section~\ref{ssc:convex},
which is the intersection of problems considered in this work
and in \cite{Nie2023convex}.
One can easily see that it is way more difficult to solve NEPs
without any convexity assumption from our discussion in
Sections~\ref{sc:alg} and \ref{ssc:convex}.

\item The goal of the method in \cite{Nie2023convex} is to find just one GNE,
and it cannot check whether the computed GNE is unique or not.
In comparison, Algorithm~\ref{ag:allKKTSDP} proposed in Section~\ref{sc: all}
aims to find more NEs.
Furthermore, when there are finitely many NEs, Algorithm~\ref{ag:allKKTSDP}
can find all NEs and check the completeness of the computed solution set,
under some general conditions.
We would like to remark that there is no other numerical method that can
achieve such computational goals for general NEPs given by polynomials,
to the best of the authors' knowledge.

\item Algorithms~\ref{ag:KKTSDP} and \ref{ag:allKKTSDP}
assume that all constraining polynomial tuples $g_i$ are nonsingular,
so that there exist polynomial expressions for Lagrange multipliers.
When the NEP is given by generic polynomials,
nonsingularity is satisfied for all $i\in[N]$.
However, polynomial Lagrange multiplier expressions typically do not exist for GNEPs.
For such cases, one may consider the corresponding
Lagrange multipliers as new variables,
but this is often computationally expensive,
especially when there are a lot of constraints.
In \cite{Nie2023convex},
rational and parametric Lagrange multiplier expressions
are studied for solving convex GNEPs.
For NEPs, when constraints are singular, rational and
parametric Lagrange multiplier expressions can also be applied to find NEs.
Nonetheless, convergence results in Theorem~\ref{tm:finiteconv}
and Corollary~\ref{cor:allNE} may no longer hold,
since there may exist NEs that are not KKT points
when polynomial expressions for Lagrange multipliers do not exist.
\end{itemize}

There is much interesting future work to do.
If there are only finitely many KKT points that are not NEs,
Algorithm~\ref{ag:KKTSDP} must terminate within finitely many loops.
This is shown in Theorem~\ref{tm:finiteconv}.
For generic NEPPs, the finiteness of KKT points
is shown in Theorem~\ref{tm:genfinite}.
However, the convergence property of Algorithm~\ref{ag:KKTSDP}
is not known when there are infinitely many KKT points.
In Example~\ref{ep:infiKKT}(ii),
there are infinitely many KKT points that are not NEs,
but Algorithm~\ref{ag:KKTSDP} is still able to get an NE in a few loops.
If there are infinitely many KKT points that are not NEs,
does Algorithm~\ref{ag:KKTSDP} still converge to find an NE?
This question is mostly open to the authors.

It is important to compute NEs efficiently for large-scale NEPs.
Even for unconstrained NEPs,
the $k$th order moment relaxation for (\ref{eq:KKTfeasopt})
is a semidefinite program with $\mcal{O}(n^{2k})$ variables.
Algorithm~\ref{ag:KKTSDP} may not be computationally practical
for solving large-scale NEPs.
Sparse polynomial optimization problems are studied in \cite{Las06sparse,Nie09sparse,Waki06,Wang21chordal,Wang21TSSOS,Wang21CSTSSOS}.
Recently, the software {\tt TSSOS} \cite{Morgan2021}
that implements the term and correlative sparse
Moment-SOS relaxations is developed.
In Algorithms~\ref{ag:KKTSDP} and \ref{ag:allKKTSDP},
polynomial optimization problems
are formulated to find NEs, and one may implement sparse Moment-SOS relaxations
for solving these polynomial optimization problems.
However, even for the NEPP where each player's
optimization problem $\mathbf{F}_i(\xmi)$ is sparse,
the polynomial optimization problem (\ref{eq:KKTfeasopt}) may not be sparse.
This is because both the polynomial expressions of Lagrange multipliers
and the KKT system may consist of dense polynomials
(see \cite{Qu2022} for more details).
Therefore, how to exploit sparsity to find NEs efficiently
for large-scale NEPs is important for future work.

Nonconvex NEPs may or may not have NEs, even if
all feasible sets are compact.
For each $i\in[N]$, let $\mcal{B}_i$ be the set of
Borel probability measures supported in $X_i$.
Define the measure function
\[
\Gamma_i(\mu_1, \ldots, \mu_N) \,  \coloneqq  \,
\int_{X_1} \cdots \int_{X_{N}} f_i(x_1,\ldots,x_N) d\mu_1 \cdots d\mu_N .
\]
The mixed strategy extension for the NEP \reff{eq:NEP} is to find $(\mu_1^*,\ldots,\mu_N^*)\in\mcal{B}_1\times\cdots\times\mcal{B}_N$ such that
\be \label{eq:mixed}
\Gamma_i(\mu_1^*, \ldots, \mu_{i-1}^*, \mu_i^*, \mu_{i+1}^*, \mu_N^*)  \, \le \,
\Gamma_i(\mu_1^*, \ldots, \mu_{i-1}^*, \mu_i, \mu_{i+1}^*, \mu_N^*)
\ee
holds for all $i\in[N]$ and for all $\mu_i\in\mcal{B}_i$.
Such a $(\mu_1^*,\ldots,\mu_N^*)$ is called a {\it mixed strategy solution}
and it always exists \cite{Glicksberg1952}.
Mixed strategy solutions to finite games are studied in
\cite{Ahmadi2020, Aubin13, DasGoldPap09, Kontogiannis2006,Nash1951,Young2014}.
The mixed strategy extensions of general continuous NEPs
are typically difficult to solve because it is a computational challenge
to do operations with measures.
However, when the functions are polynomials,
the mixed strategy extension can be equivalently expressed
in terms of moment variables.
We discuss how this can be done in the following.

For the NEPP \reff{eq:NEP}, let $a_{i,j}$
be the degree of $f_i$ in $x_j$ and let
\[ b_j \,= \, \max\{a_{1,j},\ldots,a_{N,j}\} . \]
Let $T^{(i)}$ be the $N$th order tensor such that
for all $u_j = [x_j]_{b_j}$ and $j\in[N]$,
\[
f_i(\allx) \,= \, T^{(i)}(u_1,\ldots, u_N) \,  \coloneqq
\sum_{k_1,\ldots,k_N}T^{(i)}_{k_1,\ldots,k_N}(u_1)_{k_1} \ldots (u_N)_{k_N}.
\]
Denote the set $\mcal{X}_i \coloneqq \{[x_i]_{b_i}:x_i\in X_i\}.$
Let $\cvx(\mcal{X}_i)$ be the convex hull of $\mcal{X}_i$.
For a probability measure $\mu_i\in\mcal{B}_i$,
if $u_j = \int_{X_i} [x_j]_{b_j} d\mu_i$,
then we have $u_j\in\cvx(\mcal{X}_i)$ (see \cite{HKL20,LasBk15,Lau09}).
Since $f_i$ is a polynomial, for every $(\mu_1,\ldots,\mu_N)\in\mcal{B}_1\times\ldots\times\mcal{B}_N$,
there exists
$(u_1,\dots,u_N)\in\cvx({\mcal{X}_1})\times\ldots\times\cvx({\mcal{X}_N})$
such that
\be \label{eq:int=mom}
\int_{X_1}\ldots\int_{X_N} f_i(x_1, \ldots, x_N) d\mu_1\ldots d\mu_N
\, = \,  T^{(i)}(u_1,\ldots, u_N) .
\ee
Conversely, for each
$(u_1,\dots,u_N)\in\cvx({\mcal{X}_1})\times\ldots\times\cvx({\mcal{X}_N})$,
there exist probability measures $\mu_1,\dots,\mu_N$
such that each $\mu_i\in\mcal{B}_i$ and (\ref{eq:int=mom}) holds.
Therefore, the mixed strategy extension of the NEPP \reff{eq:NEP}
is equivalent to its convex moment relaxation: find a tuple
\[
(u_1^*,\ldots,u_N^*)\in
\cvx({\mcal{X}_1})\times\ldots\times\cvx({\mcal{X}_N})
\]
such that for each $i=1,\ldots,N$,
\[
T^{(i)}(u_1^*,\ldots,u_{i-1}^*,u_i,u_{i+1}^*,\ldots,u_N^*)
\ge T^{(i)}(u_1^*,\ldots,u_N^*)
\]
for all $u_i\in\cvx(\mcal{X}_i)$.
Moreover, if each $u_i^*$ is an extreme point of $\cvx(\mcal{X}_i)$,
then one can get an NE for the original NEPP from $(u_1^*,\ldots,u_N^*)$.
We refer to \cite{LarLas12} for moment game problems,
and \cite{dresher2016polynomial,parrilo2006polynomial,stein2008separable}
for more details on mixed-strategy solutions to polynomial games.

\bigskip\noindent
{\bf Acknowledgements}
The authors would like to thank the editors and anonymous referees
for fruitful suggestions. The first author is partially supported by
the NSF grant DMS-2110780.

\appendix
\appendixpage
\addappheadtotoc
\section{Finiteness of KKT points for generic NEPPs}
\label{sc:finite}

The finiteness of KKT points implies that
Algorithms~\ref{ag:KKTSDP} and \ref{ag:allKKTSDP}
has finite termination.
In the following, we discuss the finiteness of KKT points for generic NEPPs.

After the enumeration of all possibilities of active inequality constraints,
we can generally consider the case that \reff{eq:subprob} only has equality constraints.
Consequently, the length $m_i$ of the $i$th player's constraining polynomials
can be assumed less than or equal to $n_i$, the dimension of its strategy $x_i$.
To prove the finiteness, we can ignore the sign conditions $\lmd_{i,j} \ge 0$
for Lagrange multipliers.
Then the KKT system for all players is
\be
\label{KKTsym:allplayers}
\left\{
\begin{array}{rcll}
\sum_{j=1}^{m_i} \lambda_{i,j}  \nabla_{x_i} g_{i,j}(x_i) &=&
        \nabla_{x_i} f_i(\allx) & (i \in [N]),  \\
 g_{i,j}(x_i) & = & 0 & (i \in [N], j \in  [m_i]) .
\end{array}
\right.
\ee
When the objectives $f_i$ are generic polynomials in $\allx$
and each $g_{i,j}$ is a generic polynomial in $x_i$,
we show that \reff{KKTsym:allplayers}
has finitely many complex solutions.

\begin{theorem} \label{tm:genfinite}
Let $d_{i,j}>0$, $a_{i,j} >0$ be degrees for all $i \in [N]$ and $j \in  [m_i]$.
If each $g_{i,j}$ is a generic polynomial in $x_i$ of degree $d_{i,j}$,
and each $f_i$ is a generic polynomial in $\allx$,
whose degree in $x_j$ is $a_{i,j}$, then the KKT system
\reff{KKTsym:allplayers} has finitely many complex solutions
and hence the NEP has finitely many KKT points.
\end{theorem}

\begin{proof}
For each player $i=1, \ldots, N$, denote
\[
b_i  \coloneqq  a_{i,i}-1+d_{i,1} + \cdots + d_{i,m_i} - m_i,
\]
\[
\tilde{x}_i  \coloneqq (x_{i,0},x_{i,1},\ldots,x_{i,n_i}), \quad
\tilde{x}  \coloneqq  (\tilde{x}_1, \ldots,  \tilde{x}_N  ) .
\]
The homogenization of $g_{i,j}$
is $\tilde{g}_{i,j}$, a form in $\tilde{x}_i$.
Let $\mathbb{P}^{n_i}$ be the $n_i$ dimensional projective space over the complex field.
Consider the projective varieties
\[
\mathcal{U}_i  \coloneqq  \Big \{
(\tilde{x}_1,\ldots,\tilde{x}_N)\in
\mathbb{P}^{n_1}\times\ldots\times\mathbb{P}^{n_N}:
\tilde{g}_{i}(\tilde{x}_i)=0
\Big \},\ i =1,\ldots, N,
\]
\[
\mathcal{U} \coloneqq  \mcal{U}_1 \cap \cdots \cap \mcal{U}_N.
\]
When all $g_{i,j}$ are generic polynomials in $x_i$,
the codimension of $\mcal{U}_i$ is $m_i$ (see \cite{Harris1992}),
so $\mcal{U}$ has the codimension $m_1 + \cdots + m_N$.

The $i$th player's objective $f_i$ is a polynomial in
$\allx=(x_1, \ldots, x_N)$,
we denote the multi-homogenization of $f_i(\xpi,\xmi)$ as
\[
\tilde{f}_i(\tilde{x}_i,\tilde{x}_{-i}) \coloneqq
f_i(x_1/x_{1,0},\ldots,x_N/x_{N,0})\cdot(\prod_{j=1}^N ( x_{i,0} )^{a_{i,j}}).
\]
It is a multi-homogeneous polynomial in $\tilde{\allx}$.
For each $i$, consider the determinantal variety (the $\nabla_{x_i}$
denote the gradient with respect to $x_i$)
\[
W_i \, \coloneqq \, \left\{ \allx \in  \cpx^n\left|
\begin{array}{cccc}
\rank [ \, \nabla_{x_i} f_i(\allx) & \nabla_{x_i} g_{i,1}(x_i)
           & \cdots & \nabla_{x_i} g_{i,m_i}(x_i) \, ] \le m_i
\end{array}
\right.\right\}.
\]
Its multi-homogenization is
\[\widetilde{W}_i \, \coloneqq \, \left\{ \tilde{\allx}  \left|
\begin{array}{cccc}
  \rank [ \nabla_{x_i} \tilde{f}_i(\tilde{\allx}) & \nabla_{x_i} \tilde{g}_{i,1}(\tilde{x}_i)
              & \cdots   & \nabla_{x_i} \tilde{g}_{i,m_i} (\tilde{x}_i)] \le m_i
\end{array}
\right.\right\}.
\]
The matrix in the above can be explicitly written as
\[
J_i(\tilde{x}_i, \tilde{x}_{-i}) \,  \coloneqq  \,
\left[\begin{array}{cccc}
\pt_{x_{i,1}} \tilde{f}_i(\tilde{\allx}) & \pt_{x_{i,1}} \tilde{g}_{i,1}(\tilde{x}_i)
              & \cdots   & \pt_{x_{i,1}} \tilde{g}_{i,m_i}(\tilde{x}_i)  \\
\pt_{x_{i,2}} \tilde{f}_i(\tilde{\allx}) & \pt_{x_{i,2}} \tilde{g}_{i,1}(\tilde{x}_i)
              & \cdots   & \pt_{x_{i,2}} \tilde{g}_{i,m_i} (\tilde{x}_i) \\
 \vdots &  \vdots & \ddots &  \vdots \\
\pt_{x_{i,n_i}} \tilde{f}_i(\tilde{\allx}) & \pt_{x_{i,n_i}} \tilde{g}_{i,1}(\tilde{x}_i)
              & \cdots   & \pt_{x_{i,n_i}} \tilde{g}_{i,m_i}(\tilde{x}_i)  \\
\end{array}
\right] .
\]
The $(m_i+1)$-by-$(m_i+1)$ minors of the matrix $J_i$
are homogeneous in $\tilde{x}_i$ of degree $b_i$.
They are homogeneous in $\tilde{x}_j$ of degree $a_{i,j}$, for $j \ne i$.
By \cite[Proposition~2.1]{nie2009},
when $g_{i,j}$ are generic polynomials in $x_i$,
the right $m_i$ columns of $J_i$ are linearly independent for all
$\tilde{x_i} \in \mcal{U}_i$. That is, for every $\tilde{x} \in \mcal{U}_i$,
there must exist a nonzero $m_i$-by-$m_i$ minor from the right
$m_i$ columns of $J_i$. In the following,
we consider fixed generic polynomials $g_{i,j}$.

First, we show that $\mcal{U} \cap \widetilde{W}_1$
have the codimension $n_1 + m_2 + \cdots + m_N$.
Let $\mcal{V}$ be the projective variety consisting of all
equivalent classes of the vectors
\be \label{proj:kron:all:xi}
\mathfrak{m}_1(\tilde{\allx})  \coloneqq
[\tilde{x}_1]_{b_1}^{hom} \otimes
[\tilde{x}_2]_{a_{1,2}}^{hom}  \otimes \cdots
\otimes [\tilde{x}_N]_{a_{1,N}}^{hom},
\ee
for equivalent classes of $\tilde{x} \in \mcal{U}$.
In the above, $\otimes$ denotes the Kronecker product,
$[u]_d^{hom}$ denotes the vector of
all monomials in $u$ of degrees equal to $d$.
In other words, $[u]_d^{hom}$ is the subvector of $[u]_d$
for monomials of the highest degree $d$.
Note that $\mcal{U}$ is birational to $\mcal{V}$
(consider the natural embedding $\varphi:\mcal{U}\hookrightarrow\mcal{V}$ such that $\phi(\tilde{\allx})=\mathfrak{m}_1(\tilde{\allx})$).
So $\mcal{U}$ and $\mcal{V}$ have the same codimension \cite{Shafarevich2013}.
For each subset $I \subseteq [n_1]$ of cardinality $m_1$,
we use $\det_{I} J_1$ to denote the $m_1$-by-$m_1$ minor of $J_1$
for the submatrix whose row indices are in $I$
and whose columns are the right hand side $m_1$ columns. Then
\[
\widetilde{W}_1 \quad = \bigcup_{I \subseteq [n_1], |I| = m_1}  \mcal{X}_{I}
\quad \mbox{where}
\]
\[
\mcal{X}_{I}  \coloneqq  \{ \tilde{\allx} :
\rank\, J_1(\allx) \le m_1, \, \mbox{det}_{I} J_1(\allx) \ne 0 \}.
\]
For each $I$, we have $\tilde{\allx} \in \mcal{X}_{I}$
if and only if the $(m_1+1)$-by-$(m_1+1)$ minors of $J_1$,
corresponding to the row indices $I \cup \{\ell\}$ with $\ell \in [n_1] \backslash I$,
are equal to zeros. There are totally $n_1 - m_1$ such minors.
Vanishing of these $(m_1+1)$-by-$(m_1+1)$ minors of $J_1$
gives $n_1 - m_1$ linear equations in the vector
$\mathfrak{m}_1(\tilde{\allx})$ as in \reff{proj:kron:all:xi}.
The coefficients of these linear equations are linearly
parameterized by coefficients of $f_1$. Therefore,
when $f_1$ has generic coefficients, the set
\[
\mcal{Y}_I  \coloneqq
\{\mathfrak{m}_1(\tilde{\allx}) : \tilde{\allx} \in \mcal{X}_{I}  \cap \mcal{U}  \}
\]
is the intersection of $\mcal{V}$ with hyperplanes given by $n_1 - m_1$
generic linear equations. Since $\mcal{X}_{I}  \cap \mcal{U}$
is birational to $\mcal{Y}_I$, they have the same codimension,
so the codimension of $\mcal{X}_{I}  \cap \mcal{U}$
is $n_1 + m_2 + \cdots + m_N$.
This conclusion is true for all the above subsets $I$. Since
\[
\mcal{U} \cap \widetilde{W}_1 \quad  =
\bigcup_{I \subseteq [n_1], |I| = m_1}  \mcal{X}_{I} \cap \mcal{U}  ,
\]
the codimension of $\mcal{U} \cap \widetilde{W}_1$ is equal to $n_1 + m_2 + \cdots + m_N$.

Second, we repeat the above argument to show that
\[
( \mcal{U} \cap \widetilde{W}_1 ) \cap \widetilde{W}_2
\]
has codimension $n_1 + n_2 + m_3 + \cdots + m_N$.
Let $\mcal{V}'$ be the projective variety
consisting of all equivalent classes of the vectors
\be
\mathfrak{m}_2(\tilde{\allx}) \,  \coloneqq  \,
[\tilde{x}_1]_{a_{2,1}}^{hom} \otimes
[\tilde{x}_2]_{b_2}^{hom} \otimes
[\tilde{x}_3]_{a_{2,3}}^{hom} \otimes \cdots
\otimes [\tilde{x}_N]_{a_{2,N}}^{hom}
\ee
for equivalent classes of $\tilde{x} \in \mcal{U} \cap \widetilde{W}_1$.
Note that $\mcal{U} \cap \widetilde{W}_1$ is birational to $\mcal{V}'$.
They have the same codimension. Similarly, we have
\[
\widetilde{W}_2 \quad = \bigcup_{I \subseteq [n_2], |I| = m_2}  \mcal{X}'_{I}
\quad \mbox{where}
\]
\[
\mcal{X}'_{I} \, \coloneqq  \, \{ \tilde{\allx} :
\rank\, J_2(\allx) \le m_2, \, \mbox{det}_{I} J_2(\allx) \ne 0 \}.
\]
When $f_2$ has generic coefficients, the set
\[
\mcal{Y}'_I  \coloneqq
\{\mathfrak{m}_2(\tilde{\allx}) : \tilde{\allx} \in \mcal{X}'_{I}
\cap \mcal{U} \cap \widetilde{W}_1  \}
\]
is the intersection of $\mcal{V}'$ with $n_2 - m_2$
generic hyperplanes of codimension $1$.
Since $\mcal{X}'_{I}  \cap \mcal{U} \cap \widetilde{W}_1$
is birational to $\mcal{Y}'_I$, they have the same dimension,
so the codimension of $\mcal{X}'_{I}  \cap \mcal{U} \cap \widetilde{W}_1$
is $n_1 + n_2 + m_3 + \cdots + m_N$.
This conclusion is true for all $\mcal{Y}'_I$.
Last, because
\[
\mcal{U} \cap \widetilde{W}_1 \cap \widetilde{W}_2  \quad =
\bigcup_{I \subseteq [n_2], |I| = m_2}  \mcal{X}'_{I} \cap \mcal{U} \cap \widetilde{W}_1,
\]
we know $\mcal{U} \cap \widetilde{W}_1 \cap \widetilde{W}_2$
has the codimension $n_1 + n_2 + m_3 + \cdots + m_N$.

Similarly, by repeating the above, we can eventually show that
\[
\mcal{U} \cap \widetilde{W}_1 \cap \widetilde{W}_2 \cap \cdots  \cap \widetilde{W}_N
\]
has codimension $n_1 + n_2  + \cdots + n_N$.
This implies the KKT system \reff{KKTsym:allplayers}
has codimension $n_1 + n_2 + \cdots + n_N$,
i.e., the dimension of the solution set of \reff{KKTsym:allplayers} is zero.
So, there are finitely many complex KKT points.
\end{proof}


\begin{thebibliography}{10}

\bibitem{adam2021double}
L.~Adam, R.~Hor{\v{c}}{\'\i}k, T.~Kasl, and T.~Kroupa.
\newblock Double oracle algorithm for computing equilibria in continuous games.
\newblock In {\em Proceedings of the AAAI Conference on Artificial
  Intelligence}, volume~35, pages 5070--5077, 2021.

\bibitem{Ahmadi2020}
A.~A. Ahmadi and J.~Zhang.
\newblock Semidefinite programming and {N}ash equilibria in bimatrix games.
\newblock {\em INFORMS Journal on Computing}, 33(2):607--628, 2021.

\bibitem{Aubin13}
J.-P. Aubin.
\newblock {\em Optima and equilibria: an introduction to nonlinear analysis},
  volume 140.
\newblock Springer Science \& Business Media, 2002.

\bibitem{breton2006game}
M.~Breton, G.~Zaccour, and M.~Zahaf.
\newblock A game-theoretic formulation of joint implementation of environmental
  projects.
\newblock {\em European Journal of Operational Research}, 168(1):221--239,
  2006.

\bibitem{CLOu14}
Y.~Chen, G.~Lan, and Y.~Ouyang.
\newblock Optimal primal-dual methods for a class of saddle point problems.
\newblock {\em SIAM Journal on Optimization}, 24(4):1779--1814, 2014.

\bibitem{Contreras2004}
J.~Contreras, M.~Klusch, and J.~B. Krawczyk.
\newblock Numerical solutions to {N}ash-{C}ournot equilibria in coupled
  constraint electricity markets.
\newblock {\em IEEE Transactions on Power Systems}, 19(1):195--206, 2004.

\bibitem{Couzoudis2013}
E.~Couzoudis and P.~Renner.
\newblock Computing generalized {N}ash equilibria by polynomial programming.
\newblock {\em Mathematical Methods of Operations Research}, 77(3):459--472,
  2013.

\bibitem{DasGoldPap09}
C.~Daskalakis, P.~W. Goldberg, and C.~H. Papadimitriou.
\newblock The complexity of computing a {N}ash equilibrium.
\newblock {\em SIAM Journal on Computing}, 39(1):195--259, 2009.

\bibitem{Datta10}
R.~S. Datta.
\newblock Finding all {N}ash equilibria of a finite game using polynomial
  algebra.
\newblock {\em Economic Theory}, 42(1):55--96, 2010.

\bibitem{dresher2016polynomial}
M.~Dresher, S.~Karlin, and L.~Shapley.
\newblock Polynomial games.
\newblock {\em Contributions to the Theory of Games I}, 24:161--180, 2016.

\bibitem{FacKan10OR}
F.~Facchinei and C.~Kanzow.
\newblock Generalized {N}ash equilibrium problems.
\newblock {\em Annals of Operations Research}, 175(1):177--211, 2010.

\bibitem{Farnia2020}
F.~Farnia and A.~Ozdaglar.
\newblock Do {GAN}s always have {N}ash equilibria?
\newblock In {\em International Conference on Machine Learning}, pages
  3029--3039. PMLR, 2020.

\bibitem{Glicksberg1952}
I.~L. Glicksberg.
\newblock A further generalization of the {K}akutani fixed point theorem, with
  application to {N}ash equilibrium points.
\newblock {\em Proceedings of the American Mathematical Society},
  3(1):170--174, 1952.

\bibitem{Goodfellow2014}
I.~Goodfellow, J.~Pouget-Abadie, M.~Mirza, B.~Xu, D.~Warde-Farley, S.~Ozair,
  A.~Courville, and Y.~Bengio.
\newblock Generative adversarial networks.
\newblock {\em Communications of the ACM}, 63(11):139--144, 2020.

\bibitem{GurPang09}
G.~G{\"u}rkan and J.-S. Pang.
\newblock Approximations of {N}ash equilibria.
\newblock {\em Mathematical Programming}, 117(1):223--253, 2009.

\bibitem{Harris1992}
J.~Harris.
\newblock {\em Algebraic geometry: a first course}, volume 133.
\newblock Springer Science \& Business Media, 2013.

\bibitem{HKL20}
D.~Henrion, M.~Korda, and J.~B. Lasserre.
\newblock {\em Moment-{SOS} Hierarchy, The: Lectures In Probability,
  Statistics, Computational Geometry, Control And Nonlinear Pdes}, volume~4.
\newblock World Scientific, 2020.

\bibitem{HenLas05}
D.~Henrion and J.-B. Lasserre.
\newblock Detecting global optimality and extracting solutions in {G}loptipoly.
\newblock In {\em Positive polynomials in control}, pages 293--310. Springer,
  2005.

\bibitem{GloPol3}
D.~Henrion, J.-B. Lasserre, and J.~L{\"o}fberg.
\newblock Gloptipoly 3: moments, optimization and semidefinite programming.
\newblock {\em Optimization Methods \& Software}, 24(4-5):761--779, 2009.

\bibitem{HilNie08}
C.~Hillar and J.~Nie.
\newblock An elementary and constructive solution to {H}ilbert’s 17th problem
  for matrices.
\newblock {\em Proceedings of the American Mathematical Society},
  136(1):73--76, 2008.

\bibitem{Kontogiannis2006}
S.~C. Kontogiannis, P.~N. Panagopoulou, and P.~G. Spirakis.
\newblock Polynomial algorithms for approximating {N}ash equilibria of bimatrix
  games.
\newblock In {\em International Workshop on Internet and Network Economics},
  pages 286--296. Springer, 2006.

\bibitem{krawczyk2000}
J.~B. Krawczyk and S.~Uryasev.
\newblock Relaxation algorithms to find {N}ash equilibria with economic
  applications.
\newblock {\em Environmental Modeling \& Assessment}, 5(1):63--73, 2000.

\bibitem{kroupa2023multiple}
T.~Kroupa and T.~Votroubek.
\newblock Multiple oracle algorithm to solve continuous games.
\newblock In {\em Decision and Game Theory for Security: 13th International
  Conference, GameSec 2022, Pittsburgh, PA, USA, October 26--28, 2022,
  Proceedings}, pages 149--167. Springer, 2023.

\bibitem{Kulkarni2012}
A.~A. Kulkarni and U.~V. Shanbhag.
\newblock On the variational equilibrium as a refinement of the generalized
  {N}ash equilibrium.
\newblock {\em Automatica}, 48(1):45--55, 2012.

\bibitem{LarLas12}
R.~Laraki and J.~B. Lasserre.
\newblock Semidefinite programming for min--max problems and games.
\newblock {\em Mathematical Programming}, 131(1):305--332, 2012.

\bibitem{Las01}
J.~B. Lasserre.
\newblock Global optimization with polynomials and the problem of moments.
\newblock {\em SIAM Journal on optimization}, 11(3):796--817, 2001.

\bibitem{Las06sparse}
J.~B. Lasserre.
\newblock Convergent {SDP}-relaxations in polynomial optimization with
  sparsity.
\newblock {\em SIAM Journal on Optimization}, 17(3):822--843, 2006.

\bibitem{LasBk15}
J.~B. Lasserre.
\newblock {\em An introduction to polynomial and semi-algebraic optimization},
  volume~52.
\newblock Cambridge University Press, 2015.

\bibitem{lasserre2008semidef}
J.~B. Lasserre, M.~Laurent, and P.~Rostalski.
\newblock Semidefinite characterization and computation of zero-dimensional
  real radical ideals.
\newblock {\em Foundations of Computational Mathematics}, 8(5):607--647, 2008.

\bibitem{Lau09}
M.~Laurent.
\newblock Sums of squares, moment matrices and optimization over polynomials.
\newblock In {\em Emerging applications of algebraic geometry}, pages 157--270.
  Springer, 2009.

\bibitem{Morgan2021}
V.~Magron and J.~Wang.
\newblock {TSSOS}: a {J}ulia library to exploit sparsity for large-scale
  polynomial optimization.
\newblock {\em arXiv preprint arXiv:2103.00915}, 2021.

\bibitem{maskin1999}
E.~Maskin.
\newblock {N}ash equilibrium and welfare optimality.
\newblock {\em The Review of Economic Studies}, 66(1):23--38, 1999.

\bibitem{Nash1951}
J.~Nash.
\newblock Non-cooperative games.
\newblock {\em Annals of mathematics}, pages 286--295, 1951.

\bibitem{Nedic2009}
A.~Nedi{\'c} and A.~Ozdaglar.
\newblock Subgradient methods for saddle-point problems.
\newblock {\em Journal of optimization theory and applications},
  142(1):205--228, 2009.

\bibitem{PMI2011}
J.~Nie.
\newblock Polynomial matrix inequality and semidefinite representation.
\newblock {\em Mathematics of operations research}, 36(3):398--415, 2011.

\bibitem{Nie2012}
J.~Nie.
\newblock Sum of squares methods for minimizing polynomial forms over spheres
  and hypersurfaces.
\newblock {\em Frontiers of mathematics in China}, 7(2):321--346, 2012.

\bibitem{nie2013certifying}
J.~Nie.
\newblock Certifying convergence of {L}asserre’s hierarchy via flat
  truncation.
\newblock {\em Mathematical Programming}, 142(1):485--510, 2013.

\bibitem{nie2013polynomial}
J.~Nie.
\newblock Polynomial optimization with real varieties.
\newblock {\em SIAM Journal On Optimization}, 23(3):1634--1646, 2013.

\bibitem{nie2014moment}
J.~Nie.
\newblock The $\mathcal{A}$-truncated $\mathcal{K}$-moment problem.
\newblock {\em Foundations of Computational Mathematics}, 14(6):1243--1276,
  2014.

\bibitem{nie2014optimality}
J.~Nie.
\newblock Optimality conditions and finite convergence of {L}asserre’s
  hierarchy.
\newblock {\em Mathematical Programming}, 146(1):97--121, 2014.

\bibitem{Nuc2017}
J.~Nie.
\newblock Symmetric tensor nuclear norms.
\newblock {\em SIAM Journal on Applied Algebra and Geometry}, 1(1):599--625,
  2017.

\bibitem{Nie2018}
J.~Nie.
\newblock Tight relaxations for polynomial optimization and {L}agrange
  multiplier expressions.
\newblock {\em Mathematical Programming}, 178(1):1--37, 2019.

\bibitem{Nie09sparse}
J.~Nie and J.~Demmel.
\newblock Sparse {SOS} relaxations for minimizing functions that are summations
  of small polynomials.
\newblock {\em SIAM Journal on Optimization}, 19(4):1534--1558, 2009.

\bibitem{nie2009}
J.~Nie and K.~Ranestad.
\newblock Algebraic degree of polynomial optimization.
\newblock {\em SIAM Journal on Optimization}, 20(1):485--502, 2009.

\bibitem{Nie2023convex}
J.~Nie and X.~Tang.
\newblock Convex generalized {N}ash equilibrium problems and polynomial
  optimization.
\newblock {\em Mathematical Programming}, 198:1485--1518, 2023.

\bibitem{Nie2020gs}
J.~Nie, X.~Tang, and L.~Xu.
\newblock The {G}auss--{S}eidel method for generalized {N}ash equilibrium
  problems of polynomials.
\newblock {\em Computational Optimization and Applications}, 78(2):529--557,
  2021.

\bibitem{Nie2018saddle}
J.~Nie, Z.~Yang, and G.~Zhou.
\newblock The saddle point problem of polynomials.
\newblock {\em Foundations of Computational Mathematics}, 22(4):1133--1169,
  2022.

\bibitem{NieZhang18}
J.~Nie and X.~Zhang.
\newblock Real eigenvalues of nonsymmetric tensors.
\newblock {\em Computational Optimization and Applications}, 70(1):1--32, 2018.

\bibitem{parrilo2006polynomial}
P.~A. Parrilo.
\newblock Polynomial games and sum of squares optimization.
\newblock In {\em Proceedings of the 45th IEEE Conference on Decision and
  Control}, pages 2855--2860. IEEE, 2006.

\bibitem{putinar1993positive}
M.~Putinar.
\newblock Positive polynomials on compact semi-algebraic sets.
\newblock {\em Indiana University Mathematics Journal}, 42(3):969--984, 1993.

\bibitem{Qu2022}
Z.~Qu and X.~Tang.
\newblock A correlative sparse {L}agrange multiplier expression relaxation for
  polynomial optimization.
\newblock {\em arXiv preprint arXiv:2208.03979}, 2022.

\bibitem{Ratliff2013}
L.~J. Ratliff, S.~A. Burden, and S.~S. Sastry.
\newblock Characterization and computation of local {N}ash equilibria in
  continuous games.
\newblock In {\em 2013 51st Annual Allerton Conference on Communication,
  Control, and Computing (Allerton)}, pages 917--924. IEEE, 2013.

\bibitem{Schofield2002}
N.~Schofield and I.~Sened.
\newblock Local {N}ash equilibrium in multiparty politics.
\newblock {\em Annals of Operations Research}, 109(1):193--211, 2002.

\bibitem{Swg05}
M.~Schweighofer.
\newblock Optimization of polynomials on compact semialgebraic sets.
\newblock {\em SIAM Journal on Optimization}, 15(3):805--825, 2005.

\bibitem{Shafarevich2013}
I.~R. Shafarevich.
\newblock {\em Basic Algebraic Geometry 1: Varieties in Projective Space}.
\newblock Springer Science \& Business Media, 2013.

\bibitem{stein2008separable}
N.~D. Stein, A.~Ozdaglar, and P.~A. Parrilo.
\newblock Separable and low-rank continuous games.
\newblock {\em International Journal of Game Theory}, 37(4):475--504, 2008.

\bibitem{sturm1999using}
J.~F. Sturm.
\newblock Using {SeDuMi} 1.02, a {MATLAB} toolbox for optimization over
  symmetric cones.
\newblock {\em Optimization methods and software}, 11(1-4):625--653, 1999.

\bibitem{UryRub94}
S.~Uryas'ev and R.~Y. Rubinstein.
\newblock On relaxation algorithms in computation of noncooperative equilibria.
\newblock {\em IEEE Transactions on Automatic Control}, 39(6):1263--1267, 1994.

\bibitem{Waki06}
H.~Waki, S.~Kim, M.~Kojima, and M.~Muramatsu.
\newblock Sums of squares and semidefinite program relaxations for polynomial
  optimization problems with structured sparsity.
\newblock {\em SIAM Journal on Optimization}, 17(1):218--242, 2006.

\bibitem{Wang21chordal}
J.~Wang, V.~Magron, and J.-B. Lasserre.
\newblock Chordal-{TSSOS}: a moment-{SOS} hierarchy that exploits term sparsity
  with chordal extension.
\newblock {\em SIAM Journal on Optimization}, 31(1):114--141, 2021.

\bibitem{Wang21TSSOS}
J.~Wang, V.~Magron, and J.-B. Lasserre.
\newblock {TSSOS}: A moment-{SOS} hierarchy that exploits term sparsity.
\newblock {\em SIAM Journal on Optimization}, 31(1):30--58, 2021.

\bibitem{Wang21CSTSSOS}
J.~Wang, V.~Magron, J.~B. Lasserre, and N.~H.~A. Mai.
\newblock {CS-TSSOS}: Correlative and term sparsity for large-scale polynomial
  optimization.
\newblock {\em ACM Transactions on Mathematical Software} 48(4):1--26

\bibitem{Young2014}
P.~Young and S.~Zamir.
\newblock {\em Handbook of game theory}.
\newblock Elsevier, 2014.

\end{thebibliography}
\end{document}